\documentclass[11pt,leqno]{article}

\usepackage[active]{srcltx}

\usepackage{amsfonts,latexsym,amsmath,amssymb,amsthm}

\usepackage{fullpage}
\usepackage{dsfont}
\newtheorem{theorem}{Theorem}[section]
\newtheorem{lemma}[theorem]{Lemma}

\newtheorem{proposition}[theorem]{Proposition}

\newtheorem{remark}[theorem]{Remark}
\newtheorem{claim}[theorem]{Claim}

\theoremstyle{definition}
\newtheorem{definition}[theorem]{Definition}

\theoremstyle{remark}

\newtheorem*{note*}{Note}
\numberwithin{equation}{section}
\makeatletter
\newcommand{\rank}{\mathop{\operator@font rank}}
\newcommand{\conv}{\mathop{\operator@font conv}}
\newcommand{\vol}{{\rm vol}}
\newcommand{\vrad}{{\rm vrad}}
\newcommand{\onetagright}{\tagsleft@false}
\makeatother
\newcommand{\ls}{\leqslant}
\newcommand{\gr}{\geqslant}
\renewcommand{\epsilon}{\varepsilon}

\def\irr#1{{\rm Irr}(#1)}
\def\irrr#1#2 {\irr {#1 \mid #2}}

\usepackage{color}

\usepackage{hyperref}

\usepackage{setspace}
\begin{document}
\small

\title{\bf On the deterministic interior body of random polytopes}

\medskip

\author{Minas Pafis and Natalia Tziotziou}

\date{}
\maketitle

\begin{abstract}\footnotesize Let $\{X_i\}_{i=1}^{\infty}$ be a sequence of independent copies of a random
vector $X$ in $\mathbb{R}^n$.  We revisit the question to determine the asymptotic shape of the
random polytope $K_N={\rm conv}\{X_1,\ldots ,X_N\}$ where $N>n$. We show that for any $\beta\in (0,1)$
there exists a constant $c(\beta)>0$ such that the following holds true: If $\mu $ is a Borel probability
measure on ${\mathbb R}^n$ then, for all $N\gr c(\beta)n$ we have that $K_N\supseteq T_{\beta\ln(\frac{N}{n})}(\mu)$
with probability greater than $1-\exp(-\tfrac{1}{2}N^{1-\beta}n^{\beta})$, where $T_p(\mu)$ is the convex set
of all points $x\in\mathbb{R}^n$ with half-space depth greater than or equal to $e^{-p}$. Our approach does not
require any additional assumptions about the measure $\mu$ and hence it generalizes and/or improves a sequence of
previous results. Moreover, for the class of strongly regular measures we compare the family $\{T_p(\mu)\}_{p>0}$
to other natural families of convex bodies associated with $\mu$, such as the $L_p$-centroid bodies of $\mu$
or the level sets of the Cram\'{e}r transform of $\mu$, and use this information in order to estimate the size
of a random $K_N$.
\end{abstract}

%%%%%%%%%%%%%%%%%%%%%%%%%%%%%%%%%%%%%%%%%%%%%%%%%%%%%%%%%%%%%%%%%%%%%%%%%%%%%%%%%%%%%%%%%%%%%%%%%%%%%%%%%%%%%%%%%%%%%%%%%%%%%%%
\section{Introduction}\label{section-1}
%%%%%%%%%%%%%%%%%%%%%%%%%%%%%%%%%%%%%%%%%%%%%%%%%%%%%%%%%%%%%%%%%%%%%%%%%%%%%%%%%%%%%%%%%%%%%%%%%%%%%%%%%%%%%%%%%%%%%%%%%%%%%%%

Let $\mu$ be a Borel probability measure on $\mathbb{R}^n$ and consider a sequence
$\{X_i\}_{i=1}^{\infty}$ of independent
random vectors distributed according to $\mu$. The geometry of the random polytopes $K_N={\rm conv}\{X_1,\ldots ,X_N\}$
and $S_N={\rm conv}\{\pm X_1,\ldots ,\pm X_N\}$ has been extensively studied in a series of works
by several authors. The following typical result of
Gluskin \cite{Gluskin-1988} concerns the Gaussian case, where $\mu=\gamma_n$ is the standard Gaussian measure on
$\mathbb{R}^n$ with density $(2\pi)^{-n/2}\exp(-|x|^2/2)$ (in what follows, $|\cdot|$ denotes the
Euclidean norm): For any $\beta\in (0,1)$ and $N\gr c_1(\beta)n$ one has that
\begin{equation}\label{eq:gluskin}S_N\supseteq c_2(\beta)\sqrt{\ln (eN/n)}\,B_2^n\end{equation}
with probability greater than $1-2\exp(-c_3N^{1-\beta}n^{\beta})$, where $c_i(\beta)$ are constants that depend only on $\beta$
and $c_3$ is an absolute positive constant.

This result has been extended in \cite{Litvak-Pajor-Rudelson-Tomczak-2005} to random polytopes generated by a random vector $X=(\xi_1,\ldots ,\xi_n)$ whose
coordinates are independent copies of a random variable $\xi$ with expectation $\mathbb{E}(\xi)=0$ and variance
${\rm Var}(\xi)=1$, that satisfies $\left(\mathbb{E}|\xi|^p\right)^{1/p}\ls L\sqrt{p}$ for some constant $L>0$
and every $p\gr 1$ (we then say that $\xi$ is $L$-subgaussian). For any $\beta\in (0,1)$ and $N\gr c_1(\beta,L)n$ one has that
\begin{equation}\label{eq:rudelson-et-al}S_N\supseteq c_2(\beta,L)\Big(B_{\infty}^n\cap \sqrt{\ln (eN/n)}\,B_2^n\Big)\end{equation}
with probability greater than $1-2\exp(-c_3N^{1-\beta}n^{\beta})$. A version of this fact for the special case where
the $\xi_i$'s are symmetric $\pm 1$ random variables had been previously obtained in \cite{Giannopoulos-Hartzoulaki-2002} for $N\gr n(\ln n)^2$ and with probability greater than $1-e^{-n}$.

A general study of the asymptotic shape of random polytopes when their vertices are assumed to have a log-concave
distribution was initiated by Dafnis, Giannopoulos and Tsolomitis in \cite{DGT1} and \cite{DGT2}.
Given a centered log-concave probability measure $\mu $ on
${\mathbb R}^n$, for every $N\gr n$ we consider $N$ independent random vectors $X_1,\ldots ,X_N$ distributed according to
$\mu $ and the random polytope $S_N:={\rm conv}\{ \pm X_1,\ldots ,\pm X_N\}$.
The main idea in these works was to compare $S_N$ with the $L_p$-centroid body of $\mu $
for a suitable value of $p$; roughly speaking, $S_N$ is close to the body $Z_{\ln (eN/n)}(\mu )$
with high probability.  Recall that the $L_p$-centroid bodies $Z_p(\mu)$, $p\gr 1$, are defined
through their support function $h_{Z_p(\mu)}$ given by
\begin{equation}
h_{Z_p(\mu)}(y):= \|\langle \cdot ,y\rangle \|_{L_p(\mu)} = \left(\int_{{\mathbb R}^n}|\langle x,y\rangle|^pd\mu(x)\right)^{1/p}.
\end{equation}
These bodies incorporate information about the distribution of linear functionals
with respect to $\mu$. The $L_p$-centroid bodies were introduced, under a different normalization,
by Lutwak and Zhang in \cite{Lutwak-Zhang-1997}, while in \cite{Paouris-2006} for the first time,
and in \cite{Paouris-2012} later on, Paouris used geometric properties
of them to acquire detailed information about the distribution of the Euclidean norm with
respect to $\mu $. The starting observation in \cite{DGT1} was that the deterministic body
$\sqrt{p}\,B_2^n$ that appears in \eqref{eq:gluskin} is approximately equal to $Z_p(\gamma_n)$ and
the deterministic body $B_{\infty}^n\cap \sqrt{p}\,B_2^n$ that appears in \eqref{eq:rudelson-et-al} is approximately
equal to $Z_p(\nu_n)$, where $\nu_n$ is the uniform measure on the cube.
It was proved in \cite{DGT1} that, given any centered log-concave probability measure $\mu $
on ${\mathbb R}^n$ and any $cn\ls N\ls e^n$, the random polytope $S_N$ defined by $N$ independent random vectors
$X_1,\ldots ,X_N$ which are distributed according to $\mu $ satisfies the inclusion
\begin{equation}\label{eq:DGT}S_N\supseteq c_1Z_{\ln (eN/n)}(\mu )\end{equation}
with probability greater than
\begin{equation*}1-2\exp\left (-c_3N^{1-\beta }n^{\beta }\right
)-{\mathbb P}(\|\Gamma :\ell_2^n\to\ell_2^N\|\gr \gamma\sqrt{N}) \gr
1-\exp (-c_0\gamma\sqrt{N}),\end{equation*} where
$\Gamma:\ell_2^n\to\ell_2^N$ is the random operator
$\Gamma(y)=(\langle X_1, y\rangle, \ldots \langle X_N, y\rangle)$. The last inequality follows from the estimate
\begin{equation*}{\mathbb P}(\|\Gamma :\ell_2^n\to\ell_2^N\|\gr \gamma
\sqrt{N})\ls\exp (-c\gamma\sqrt{N})\end{equation*} for all $N\gr
\gamma n$, that has been obtained in \cite{ALPT}.

This approach was further extended to random polytopes with vertices that have an arbitrary symmetric
distribution $\mu$ on $\mathbb{R}^n$. Gu\'{e}don, Krahmer, K\"{u}mmerle, Mendelson and Rauhut introduced
in \cite{GKKMR-2022} (see also \cite{Mendelson-2020} and the earlier work \cite{GLT-2020} which was
the motivation for \cite{Mendelson-2020} and \cite{GKKMR-2022}) the family of sets $\{U_p(\mu)\}_{p\gr 1}$ defined by
\begin{equation}\label{eq:def-U}U_p(\mu)=\{y\in \mathbb{R}^n:\;\mu(\{x\in\mathbb{R}^n:\langle x,y\rangle\gr 1\})
\ls \exp(-p)\}\end{equation}
and showed that, under some assumptions on $\mu$, the random polytopes $S_N$ and $K_N$ contain
$\frac{1}{2}(U_p(\mu))^{\circ}$ with probability close to $1$, where $A^{\circ}$ denotes the polar set of $A$.
Their assumption on $\mu$ is that there exists a norm $\|\cdot\|$ on $\mathbb{R}^n$ and some positive constants
$\gamma,\delta,r$ and $L$ such that the {\it small ball condition}
$\mu(\{x\in\mathbb{R}^n:|\langle x,y\rangle |\gr\gamma\|y\|\})\gr\delta $ and the {\it $L_r$-condition}
$\big(\mathbb{E}_{\mu}|\langle \cdot ,y\rangle |^r\big)^{1/r}\ls L\|y\|$ are satisfied for every $y\in\mathbb{R}^n$.
The main result in \cite{GKKMR-2022} asserts that if $\mu$ is a symmetric Borel probability measure on $\mathbb{R}^n$
that satisfies a small ball condition and an $L_r$-condition with constants $\gamma,\delta, r$ and $L$ for some
norm on $\mathbb{R}^n$ then, for any $0<\beta<1$ there exists a positive constant $c_0:=c(\beta,\delta,r,L/\gamma)$ such
that if $N\gr c_0n$ and $p=\beta\ln(eN/n)$ we have
\begin{equation}\label{eq:GKKMR}S_N\supseteq\frac{1}{2}(U_p(\mu))^{\circ}\end{equation}
with probability greater than $1-2\exp(-c_1N^{1-\beta}n^{\beta})$, where $c_1>0$ is an absolute constant.
It is explained in \cite{GKKMR-2022} that the same inclusion holds for the random polytope $K_N$
and that it implies \eqref{eq:DGT}, with the improved and optimal probability estimate stated above,
when $\mu$ is a centered log-concave probability measure (we shall discuss and further study the relation
between $U_p(\mu)$ and $Z_p(\mu)$ in Sections~\ref{section-3} and~\ref{section-4}).

Let now $\mu$ be an arbitrary Borel probability measure on $\mathbb{R}^n$. For any $x\in {\mathbb R}^n$ we denote by ${\cal H}(x)$ the set
of all closed half-spaces $H^+$ of ${\mathbb R}^n$ containing $x$. The function
$$\varphi_{\mu }(x)=\inf\{\mu (H^+):H^+\in {\cal H}(x)\}$$
(introduced by Tukey in \cite{Tukey-1975}) is the {\it Tukey half-space depth function} of $\mu$.
Note that the infimum in the definition of $\varphi_{\mu}(x)$
is determined by those closed half-spaces $H^+$ for which $x$ lies on the boundary ${\rm bd}(H^+)$ of $H^+$.
It is useful to note that the half-space depth function $\varphi_{\mu}$ attains its maximum
and $\max(\varphi_{\mu})\gr\frac{1}{n+1}$. Every point $x$ that satisfies $\varphi_{\mu}(x)=\max(\varphi_{\mu})$ is called a
center point for $\mu$.

Hayakawa, Lyons and Oberhauser showed in \cite{HLO-2023} that, in the case where $\mu$ is assumed symmetric,
an inclusion which is essentially equivalent to \eqref{eq:GKKMR}
continues to hold even if we do not assume the small ball condition and the $L_r$-condition for $\mu$. They considered
the family $\{T_p(\mu)\}_{p>0}$ of the level sets of the Tukey half-space depth $\varphi_{\mu}$ of $\mu$
which are defined by
$$T_p(\mu)=\{x\in\mathbb{R}^n:\varphi_{\mu}(x)\gr e^{-p}\}$$
and proved the following.

\begin{theorem}[Hayakawa-Lyons-Oberhauser]\label{th:HLO}Let $\mu$ be a symmetric Borel probability measure
on $\mathbb{R}^n$. Let $0<\beta<1$ and set $p=\beta\ln(eN/n)$. Then, if $N\gr (12e^{\beta})^{\frac{1}{1-\beta}}n$
we have that \begin{equation}\label{eq:HLO}K_N\supseteq \frac{1}{2}T_p(\mu)\end{equation}
with probability greater than $1-2\exp(-ce^{-\beta}N^{1-\beta}n^{\beta})$, where $c>0$ is an absolute constant.
\end{theorem}

It was also proved in \cite{HLO-2023} that if $\mu$ is a symmetric Borel probability measure then
$T_p(\mu)$ is ``essentially" the polar set of $U_p(\mu)$. Therefore, Theorem~\ref{th:HLO} implies
the main result of \cite{GKKMR-2022}. Namely, for \eqref{eq:GKKMR} to hold, we do not have
to assume the small ball condition or the $L_r$-condition for $\mu$. In Proposition~\ref{prop:U-T}
we show that $(U_p(\mu))^{\circ}\subseteq T_p(\mu)$ for any (not necessarily symmetric) probability measure
which is ``translated" in such a way that the origin is a center point.

The proof of Theorem~\ref{th:HLO} involves some sharp estimates connecting the Tukey half-space depth function
$\varphi_{\mu}$ with the parameters $p_{N,\mu}(x)=\mathbb{P}(x\in K_N)$
and $N_{\mu}(x)=\min\left\{N\in\mathbb{N}:p_{N,\mu}(x)\gr\tfrac{1}{2}\right\}$.
It is proved in \cite[Proposition~13]{HLO-2023} that
\begin{equation}\label{eq:HLO-1}1-p_{N,\mu}(x)\ls\left(\frac{N\varphi_{\mu}(x)}{n}
\exp\left\{\left(\frac{1}{\varphi_{\mu}(x)}\ln\frac{1}{1-\varphi_{\mu}(x)}\right)
\left(1+\varphi_{\mu}(x)-\frac{N\varphi_{\mu}(x)}{n}\right)\right\}\right)^n\end{equation}
for every $N\gr n/\varphi_{\mu}(x)$. Since $\frac{1}{t}\ln\frac{1}{1-t}\gr 1$ for all $t\in (0,1)$,
if we assume that $\frac{N}{n}\gr \frac{1}{\varphi_{\mu}(x)}+1$ then we may use the simpler bound
\begin{equation}\label{eq:HLO-2}1-p_{N,\mu}(x)\ls\left(\frac{N\varphi_{\mu}(x)}{n}
\exp\left(1+\varphi_{\mu}(x)-\frac{N\varphi_{\mu}(x)}{n}\right)\right)^n\end{equation}
This is the main estimate in \cite{HLO-2023}, which is essential for the proof of Theorem~\ref{th:HLO}.
Another main consequence of \eqref{eq:HLO-1} is the fact that one can relate $N_{\mu}(x)$ and $\varphi_{\mu}(x)$
as follows:
\begin{equation}\label{eq:HLO-3}N_{\mu}(x)\ls\left\lceil\frac{3n}{\varphi_{\mu}(x)}\right\rceil, \end{equation}
an inequality that holds true for any $n$-dimensional probability measure $\mu$ and any $x$ (see \cite[Theorem~16]{HLO-2023}).

\smallskip

\textbf{Our contribution.} Our first main result states that one can have the assertion of
Theorem~\ref{th:HLO} without the assumption of symmetry for the measure $\mu$.

\begin{theorem}\label{th:deterministic}For any $\beta\in (0,1)$ there exists a constant $c(\beta)>0$ such that
the following holds true. If $\mu $ is a Borel probability measure on ${\mathbb R}^n$ then, for every $N\gr c(\beta)n$ we have that
$$K_N\supseteq T_{\beta\ln(\frac{N}{n})}(\mu)$$
with probability greater than $1-\exp(-\tfrac{1}{2}N^{1-\beta}n^{\beta})$.
\end{theorem}

Our proof of Theorem~\ref{th:deterministic}, which is valid in full generality, is based on the
$\varepsilon$-net theorem of Haussler and Welzl \cite{Haussler-Welzl} which we introduce and discuss
in Section~\ref{section-3}. We use a sharp version of the $\epsilon$-net theorem that has been used by
Nasz\'{o}di for questions related to random polytopes. It is due to Koml\'{o}s, Pach
and Woeginger, see \cite[Theorem~3.2]{Komlos-Pach-Woeginger}. We will actually exploit a
simplified form of this theorem that appears in \cite[Lemma~3.2]{Naszodi}.

It is interesting to note that Theorem~\ref{th:HLO} is in a sense equivalent to the inequality \eqref{eq:HLO-3}. We present
a very simple argument, which is again based on the $\epsilon$-net theorem, showing that an inequality similar to \eqref{eq:HLO-3},
but slightly weaker, holds true in full generality.

\begin{theorem}\label{th:N-varphi}
Let $\mu$ be a Borel probability measure on $\mathbb{R}^n$. For every $x\in\mathbb{R}^n$ we have that
$$N_{\mu}(x)\ls\frac{cn}{\varphi_{\mu}(x)}\left(1+\ln\left(1/\varphi_{\mu}(x)\right)\right)$$
where $c>0$ is an absolute constant.
\end{theorem}

Theorem~\ref{th:deterministic} is the most general in the series of results that we have discussed and hence we
may deduce very general versions of all the consequences that have appeared in earlier works.
Concrete applications of the theorem require computing and estimating the size of the bodies
$T_p(\mu)$ for an individual probability measure $\mu$, which is not a simple task.
In Section~\ref{section-4} we give equivalent and more convenient descriptions
of the bodies $T_p(\mu)$ under additional assumptions on the measure $\mu$. A natural class of measures, that was already
considered in \cite{GKKMR-2022}, is the class of $\alpha$-regular
or $\alpha$-strongly regular measures, which is broader than the class of centered log-concave probability
measures. We say that $\mu$ is $\alpha $-regular if $y\mapsto \|\langle \cdot,y\rangle_+\|_{L^{1}(\mu)}$ is
bounded on $S^{n-1}$ and
$$\|\langle \cdot,y\rangle_+\|_{L^{2p}(\mu)}\ls 2\alpha \,\|\langle \cdot,y\rangle_+\|_{L^{p}(\mu)}$$
for every $y\in\mathbb{R}^n$ and any $p\gr 1$, and that $\mu$ is $\alpha $-strongly regular if it
satisfies the stronger condition
$$\|\langle \cdot ,y\rangle_+\|_{L^{q}(\mu)}
\ls \frac{\alpha q}{p}\,\|\langle \cdot,y\rangle_+\|_{L^{p}(\mu)}$$
for every $y\in\mathbb{R}^n$ and any $q>p\gr 1$, where $a_+=\max\{a,0\}$. We introduce the family $\{Z_p^+(\mu)\}_{p\gr 1}$
of nonsymmetric $L_p$-centroid bodies and the family $\{B_p(\mu)\}_{p>0}$ of the level sets
of the Cram\'{e}r transform of an $\alpha$-strongly regular measure $\mu$ on $\mathbb{R}^n$, and show that
the three families are equivalent up to absolute constants.

\begin{theorem}\label{th:regular-families}
Let $\mu$ be an $\alpha$-regular Borel probability measure on $\mathbb{R}^n$. If $0$ is a center point for
$\mu$ then for every $p\gr \frac{1}{2\ln (2e\alpha )}\ln (n+1)$ we have that
$$Z_p^+(\mu)\subseteq 2T_{2\ln(2e\alpha )p}(\mu).$$
Moreover, if $\mu$ is $\alpha $-strongly regular then
$T_p(\mu)\subseteq B_p(\mu) \subseteq 2Z_{c_1\alpha p}^+(\mu)$ for every $p\gr 1$, where $c_1>0$ is an
absolute constant.
\end{theorem}

We close Section~\ref{section-4} with a brief discussion of analogous, more precise, results when the measure $\mu$ is assumed
log-concave or $s$-concave (see Section~\ref{section-2} for background information).

\smallskip

Computing the volume of the $L_p$-centroid bodies of a probability measure $\mu$ is possible if we assume
that $\mu$ has a bounded density. We obtain a lower bound, using the family $\{K_p(\mu)\}_{p>0}$
of K.~Ball's star bodies associated with $\mu$ in order to reduce ourselves to the same question for
a star body from this family, and then employing the $L_p$-affine isoperimetric
inequality of Lutwak, Yang and Zhang \cite{Lutwak-Yang-Zhang-2000}, and more precisely
its refined version by Haberl and Schuster from \cite{Haberl-Schuster-2009}.

We say that a Borel probability measure $\mu$ on $\mathbb{R}^n$ belongs to the class $\mathcal{P}_n$ if
it has a bounded density $f_\mu$, the set $K_{f_\mu}=\{f_\mu>0\}$ is convex and has $0$ in its interior,
and the restriction of $f_\mu$ to $K_{f_\mu}$ is continuous. For the next theorem we need to assume
that $\mu\in\mathcal{P}_n$.

\begin{theorem}\label{th:volume}
Let $\mu$ be a probability measure on $\mathbb{R}^n$ which belongs to the class $\mathcal{P}_n$. Then,
\begin{equation*}\vol_n(Z_p^+(\mu))^{1/n}\gr c\|f_{\mu}\|_{\infty}^{-1/n}\sqrt{p/n}\end{equation*}
for every $1\ls p\ls n$, where $Z_p^+(\mu)$ is the nonsymmetric $L_p$-centroid body of $\mu$ and $c>0$ is an absolute constant.
\end{theorem}

Theorem~\ref{th:regular-families} and Theorem~\ref{th:volume} allow
us to state and prove a general theorem about the asymptotic shape of random
polytopes with independent vertices that have $\alpha $-regular distribution.

\begin{theorem}\label{th:intro-regular-size}Let $\beta\in (0,1)$ and $\alpha\gr\frac{1}{2}$. Set
$r(\alpha,\beta):=\frac{2\ln(2e\alpha)}{\beta}$ and $t(\alpha,\beta):=\frac{\beta}{2\ln(2e\alpha)}$.
If $\mu $ is an $\alpha$-regular probability measure on $\mathbb{R}^n$ which has $0$ as a center point then
for any $(n+1)^{1+r(\alpha,\beta)}\ls N\ls e^n$
we have that
$$K_N\supseteq \frac{1}{2} Z_{t(\alpha,\beta)\ln\left(\frac{N}{n}\right)}^+(\mu)$$
with probability greater than $1-\exp(-\tfrac{1}{2}N^{1-\beta}n^{\beta})$.
Moreover, if $\mu$ also belongs to the class  $\mathcal{P}_n$, then
for any $(n+1)^{1+r(\alpha,\beta)}\ls N\ls e^n$ we have that
$$\vol_n(K_N)^{1/n}\gr c\sqrt{t(\alpha,\beta)}\|f_{\mu}\|_{\infty}^{-1/n}\frac{\sqrt{\ln(N/n)}}{\sqrt{n}}$$
with the same probability, where $c>0$ is an absolute constant.
\end{theorem}

In particular, Theorem~\ref{th:intro-regular-size} is valid for every centered log-concave probability
measure $\mu$ on $\mathbb{R}^n$. In this case, or even in the case of a centered $s$-concave measures with $s \in \left[-\frac{1}{2n+1},0\right)$, we obtain a more precise result: for any $c_1(\beta)n \ls N \ls e^n$ we have that
$$\vol_n(K_N)^{1/n}\gr c_2 \sqrt{\beta} \|f_{\mu}\|_{\infty}^{-1/n}\frac{\sqrt{\ln(N/n)}}{\sqrt{n}}$$
with probability greater than $1-\exp(-\tfrac{1}{2}N^{1-\beta}n^{\beta})$, where $c_1(\beta)>0$ is a constant depending only on $\beta$ and $c_2>0$ is an absolute constant.

Compared with previous results (from \cite{DGT1} or \cite{GKKMR-2022})
this particular case of the theorem has the following three advantages:
\begin{enumerate}
\item[(i)] In contrast to \cite{GKKMR-2022}, the measure $\mu$ is not assumed
symmetric; it is just centered.
\item[(ii)] In contrast to \cite{DGT1} the statement is about the random polytope $K_N$ and not about its
absolute convex hull $S_N$.
\item[(iii)] The probability estimate is optimal with respect to $\beta$, $N$ and $n$, the one which is obtained
in \cite{GKKMR-2022}.
\end{enumerate}
Under the log-concavity assumption, a reverse inequality is also possible for the expectation
of $\vol_n(K_N)^{1/n}$. Extending ideas from~\cite{DGT1} we obtain the following ``asymptotic formula"
in the range $n^2\ls N\ls e^n$.

\begin{theorem}\label{th:log-concave-expectation}Let $\mu $ be a centered log-concave probability measure
on $\mathbb{R}^n$. Then, for every $n^2\ls N\ls e^n$ we have that
$$\mathbb{E}\big(\vol_n(K_N)^{1/n}\big)\approx \mathbb{E}\big(\vol_n(S_N)^{1/n}\big)
\approx \frac{\sqrt{\ln N}}{\sqrt{n}}.$$
\end{theorem}

We refer to Schneider's book \cite{Schneider-book} for classical facts from the Brunn-Minkowski theory and to the book
\cite{AGA-book} for an exposition of the main results from asymptotic convex geometry. We also refer
to \cite{BGVV-book} for more information on isotropic convex bodies and log-concave probability measures.

%%%%%%%%%%%%%%%%%%%%%%%%%%%%%%%%%%%%%%%%%%%%%%%%%%%%%%%%%%%%%%%%%%%%%%%%%%%%%%%%%%%%%%%%%%%%%%%%%%%%%%%%%%%%%%%%
\section{Notation and background information}\label{section-2}
%%%%%%%%%%%%%%%%%%%%%%%%%%%%%%%%%%%%%%%%%%%%%%%%%%%%%%%%%%%%%%%%%%%%%%%%%%%%%%%%%%%%%%%%%%%%%%%%%%%%%%%%%%%%%%%%

First we recall some basic notation and definitions from convex geometry. We work in ${\mathbb R}^n$, which is equipped with the standard inner product  $\langle\cdot ,\cdot\rangle $. We denote the corresponding Euclidean norm by $|\cdot |$, and write $B_2^n$ for the Euclidean unit ball, and $S^{n-1}$ for the unit sphere. Volume in $\mathbb{R}^n$ is denoted by $\vol_n$. We write $\omega_n$ for the volume of $B_2^n$ and $\sigma $ for the rotationally invariant probability measure on $S^{n-1}$.  For any $u\in\mathbb{R}^n\setminus\{0\}$ and $\alpha\in\mathbb{R}$ we
define
\begin{align*}
H(\alpha,u) &=\{y\in\mathbb{R}^n:\langle y,u\rangle =\alpha\},\\
H^+(\alpha,u) &=\{y\in\mathbb{R}^n:\langle y,u\rangle \gr\alpha\},\\
H^-(\alpha,u) &=\{y\in\mathbb{R}^n:\langle y,u\rangle \ls\alpha\}.
\end{align*}

The letters $c,c^{\prime }, c_1, c_2$ etc. denote absolute positive constants whose value may change from line to line. Whenever we
write $a\approx b$, we mean that there exist absolute constants $c_1,c_2>0$ such that $c_1a\ls b\ls c_2a$.  Also if $A,B\subseteq
\mathbb R^n$ we shall write $A\approx B$ if there exist absolute constants $c_1, c_2>0$ such that $c_{1}A\subseteq B\subseteq c_{2}A$.

A non-empty set $A \subseteq \mathbb{R}^n$ is called {\it star-shaped at the origin} if for every $x \in A$ and any $0 \ls \lambda \ls 1$, we have that $\lambda x \in A$. The radial function $\rho_A:\mathbb R^n\setminus \{0\}\to \mathbb R^+$ of a star-shaped at the origin
set $A$ is the function $\rho_{A}(x)=\sup\{t\gr 0 : tx\in A\}$. If $\rho_A$ is continuous and positive on $S^{n-1}$, then we say that
$A$ is a {\it star body}. The polar set of a star-shaped at the origin set $A$ is
\begin{equation}A^\circ=\{x\in \mathbb R^n : \langle x,y\rangle \ls 1 \;\hbox{for all}\; y\in A\}.\end{equation}

A {\it convex body} in ${\mathbb R}^n$ is a compact convex subset $K$ of ${\mathbb R}^n$ with non-empty interior.
We say that $K$ is symmetric if $K=-K$, and that $K$ is centered if its barycenter ${\rm bar}(K)$ is at the origin, i.e. if
\begin{equation*}\int_K\langle x,\xi\rangle \, dx=0\end{equation*} for every $\xi\in S^{n-1}$.
The support function of $K$ is defined for every $y\in {\mathbb R}^n$ by
$h_K(y)=\max \{\langle x,y\rangle :x\in K\}$.
The volume radius of $K$ is the quantity
\begin{equation*}\vrad(K)=\left (\frac{\vol_n(K)}{\vol_n(B_2^n)}\right )^{1/n}.\end{equation*}
Note that $\omega_n^{1/n}\approx 1/\sqrt{n}$. Therefore, $\vrad(K)\approx \sqrt{n}\,\vol_n(K)^{1/n}$.
We also write $\overline{K}$ for the multiple of $K\subseteq \mathbb R^n$ that has volume $1$; in other words, $\overline{K}:=\vol_n(K)^{-1/n}K$.

We say that a Borel probability measure $\mu $ on $\mathbb{R}^n$ is symmetric if $\mu (-B)=\mu (B)$ for every Borel subset $B$
of ${\mathbb R}^n$ and that $\mu $ is centered if the barycenter ${\rm bar}(\mu)=\int_{\mathbb{R}^n}x\,d\mu(x)$ of $\mu$ is at the origin, i.e.
\begin{equation*}\int_{\mathbb R^n} \langle x, \xi \rangle d\mu(x)= 0\end{equation*}
for all $\xi\in S^{n-1}$. Moreover, we say that $\mu$ is full-dimensional if $\mu(H)<1$ for every
hyperplane $H$ in ${\mathbb R}^n$.

A Borel measure $\mu$ on $\mathbb R^n$ is called log-concave if it is full-dimensional and
$$\mu(\lambda A+(1-\lambda)B) \gr \mu(A)^{\lambda}\mu(B)^{1-\lambda}$$
for any pair of compact sets $A,B$ in ${\mathbb R}^n$ and any $\lambda \in (0,1)$. Borell \cite{Borell-1974} has proved that, under these assumptions, $\mu $ has a log-concave density $f_{{\mu }}$. Recall that a function $f:\mathbb R^n \rightarrow [0,\infty)$ is called log-concave if its support $\{f>0\}$ is a convex set in ${\mathbb R}^n$ and the restriction of $\ln{f}$ to it is concave. The Brunn-Minkowski inequality implies that if $K$ is a convex body in $\mathbb R^n$ then the indicator function $\mathds{1}_{K} $ of $K$ is the density of a log-concave measure, the Lebesgue measure on $K$. If $\mu$ is symmetric then $f_{\mu}$ is even and it follows that $\|f_{\mu}\|_{\infty}=f_{\mu}(0)$. On the other hand, Fradelizi \cite{Fradelizi-1997} has shown that if $\mu $ is a centered log-concave probability measure then $\|f_{\mu }\|_{\infty }\ls e^nf_{\mu }(0)$.

A consequence of Borell's lemma \cite[Lemma~2.4.5]{BGVV-book} is the fact that for any seminorm
$f:\mathbb{R}^n\to \mathbb{R}$ and any $q>p\gr 1$, we have the Kahane-Khintchine inequalities
\begin{equation}\label{eq:kahane}\|f\|_{L^p(\mu)}\ls \|f\|_{L^q(\mu)}\ls C\frac{q}{p}\|f\|_{L^p(\mu)},\end{equation}
where $C>0$ is an absolute constant (see \cite[Theorem~2.4.6]{BGVV-book} for a proof).

For any log-concave measure $\mu$ on ${\mathbb R}^n$ with density $f_{\mu}$, we define the isotropic constant of $\mu $ by
\begin{equation*}
L_{\mu }:=\left (\frac{\sup_{x\in {\mathbb R}^n} f_{\mu} (x)}{\int_{{\mathbb
R}^n}f_{\mu}(x)dx}\right )^{\frac{1}{n}} \big(\det \textrm{Cov}(\mu)\big)^{\frac{1}{2n}},\end{equation*}
where $\textrm{Cov}(\mu)$ is the covariance matrix of $\mu$ with entries
\begin{equation*}\textrm{Cov}(\mu )_{ij}:=\frac{\int_{{\mathbb R}^n}x_ix_j f_{\mu}
(x)\,dx}{\int_{{\mathbb R}^n} f_{\mu} (x)\,dx}-\frac{\int_{{\mathbb
R}^n}x_i f_{\mu} (x)\,dx}{\int_{{\mathbb R}^n} f_{\mu}
(x)\,dx}\frac{\int_{{\mathbb R}^n}x_j f_{\mu}
(x)\,dx}{\int_{{\mathbb R}^n} f_{\mu} (x)\,dx}.\end{equation*}
A log-concave probability measure $\mu $ on ${\mathbb R}^n$
is called isotropic if it is centered and $\textrm{Cov}(\mu )=I_n$,
where $I_n$ is the identity $n\times n$ matrix. Note that if $\mu$ is isotropic then $L_{\mu }=\|f_{\mu }\|_{\infty }^{1/n}$.
It is not hard to check that for every log-concave probability measure $\mu $ on $\mathbb{R}^n$ there exists
an invertible affine transformation $G$ such that the log-concave probability measure $G_{\ast }\mu $ defined by
\begin{equation*}G_{\ast }\mu (B)=\mu (G^{-1}(B))\end{equation*}is isotropic,
and $L_{G_{\ast}\mu}=L_{\mu}$.

The hyperplane conjecture asks if there exists an absolute constant $C>0$ such that
\begin{equation*}L_n:= \max\{ L_{\mu }:\mu\ \hbox{is an isotropic log-concave probability measure on}\ {\mathbb R}^n\}\ls C\end{equation*}
for all $n\gr 2$. The classical estimates $L_n\ls c\sqrt[4]{n}\ln n$ by Bourgain \cite{Bourgain-1991}
and, fifteen years later, $L_n\ls c\sqrt[4]{n}$ by Klartag \cite{Klartag-2006} remained the best known until
2020. In a breakthrough work, Chen \cite{C} proved that for any $\epsilon >0$
one has $L_n\ls n^{\epsilon}$ for all large enough $n$. This development was the starting point for
a series of important works that culminated in the final affirmative answer to the problem
by Klartag and Lehec \cite{KL} after an important contribution of Guan \cite{Guan}. Soon afterwards, one more
proof of the hyperplane conjecture was presented by Bizeul \cite{Bizeul}.

We will also consider $s$-concave measures. We say that a measure $\mu $ on
${\mathbb R}^n$ is $s$-concave for some $-\infty\ls s\ls 1/n$ if
\begin{equation}\label{eq:s-concave-measure-1}\mu ((1-\lambda )A+\lambda B)\gr
((1-\lambda )\mu^{s}(A)+\lambda \mu^{s}(B))^{1/s}\end{equation}
for any pair of compact sets $A,B$ in ${\mathbb R}^n$ with $\mu (A)\mu (B)>0$ and any $\lambda\in (0,1)$.
We can also consider the limiting cases $s =0$, where the right-hand side in \eqref{eq:s-concave-measure-1}
should be interpreted as $\mu(A)^{1-\lambda }\mu (B)^{\lambda }$, and $s=-\infty $, where the right-hand side in \eqref{eq:s-concave-measure-1} becomes $\min\{\mu (A),\mu (B)\}$. Note that $0$-concave measures are the log-concave
measures and that if $\mu $ is $s$-concave and $s^{\prime}\ls s$ then $\mu $ is also $s^{\prime}$-concave.
We shall extend some of our results to $s$-concave measures with $s\in (-\infty,0)$. These classes of
measures are strictly larger than the class of log-concave measures.

A function $f:\mathbb R^n\to [0,\infty)$ is called $\gamma $-concave for some $\gamma\in [-\infty ,\infty ]$ if
\begin{equation*}f((1-\lambda )x+\lambda y)\gr ((1-\lambda )f^{\gamma }(x)+\lambda f^{\gamma }(y))^{1/\gamma }\end{equation*}
for all $x,y\in {\mathbb R}^n$ with $f(x)f(y)>0$ and all $\lambda\in (0,1)$. One can also
define the cases $\gamma =0,+\infty $ appropriately. Borell \cite{Borell-1975}
showed that if $\mu $ is a measure on ${\mathbb R}^n$ and the affine subspace $F$ spanned by the support
${\rm supp}(\mu )$ of $\mu $ has dimension ${\rm dim}(F)=n$ then for every $-\infty \ls s <1/n$ we have that
$\mu $ is $s$-concave if and only if it has a non-negative density $f\in L_{{\rm loc}}^1({\mathbb R}^n,dx)$
which is $-\frac{1}{\alpha}$-concave, where $-\frac{1}{\alpha}=\frac{s}{1-sn}$, or equivalently,
$\alpha =n-\frac{1}{s}>n$.

%%%%%%%%%%%%%%%%%%%%%%%%%%%%%%%%%%%%%%%%%%%%%%%%%%%%%%%%%%%%%%%%%%%%%%%%%%%%%%%%%%%%%%%%%%%%%%%%%%%%%%%%%%%%%%%%
\section{Deterministic interior body of random polytopes}\label{section-3}
%%%%%%%%%%%%%%%%%%%%%%%%%%%%%%%%%%%%%%%%%%%%%%%%%%%%%%%%%%%%%%%%%%%%%%%%%%%%%%%%%%%%%%%%%%%%%%%%%%%%%%%%%%%%%%%%

Let $\mu $ be a Borel probability measure on ${\mathbb R}^n$. Recall from the introduction that the Tukey half-space
depth with respect to $\mu$ is the function defined for any $x\in {\mathbb R}^n$ by
$\varphi_{\mu }(x)=\inf\{\mu (H^+):H^+\in {\cal H}(x)\}$,
where ${\cal H}(x)$ is the set of all closed half-spaces $H^+$ of ${\mathbb R}^n$ containing $x$.
It is known that $\varphi_{\mu}$ attains its maximum (see Lemma~\ref{lem:T} below) and
\begin{equation}\label{eq:T1}\frac{1}{n+1}\ls \max(\varphi_{\mu})\ls \frac{1}{2}\big(1+\sup\{\mu(\{x\}):x\in\mathbb{R}^n\}\big).\end{equation}
The left-hand side inequality is a well-known fact (see Rado~\cite{Rado}):
for every Borel probability measure $\mu$ on $\mathbb{R}^n$ there exists a point $x_0\in {\rm supp}(\mu)$ such that
$$\varphi_{\mu}(x_0)\gr \frac{1}{n+1}.$$
For the right-hand side inequality of \eqref{eq:T1}  see \cite[Lemma~1]{RR-1999}.

For every $p>0$ we define the set
$$T_p(\mu)=\{x\in \mathbb{R}^n:\varphi_{\mu}(x)\gr e^{-p}\}.$$
Note that the half-space depth is invariant under affine transformations $G$ of full rank; one can easily see that
$\varphi_{G_{\ast}\mu}(G(x))=\varphi_{\mu}(x)$ for every $x\in\mathbb{R}^n$. It follows that
$$T_p(G_{\ast}\mu)=G(T_p(\mu))$$
for all $p>0$.

In the next lemma we collect the basic properties of the family $\{T_p(\mu)\}_{p>0}$ and provide references
for their proofs.

\begin{lemma}\label{lem:T}Let $\mu $ be a Borel probability measure on ${\mathbb R}^n$. Define $p(\mu)$
by the equation
\begin{equation}\label{eq:def-pmu}\max(\varphi_{\mu})=e^{-p(\mu)}.\end{equation}
Then, $T_p(\mu)$ is a non-empty compact convex set for all $p\gr p(\mu)$.
\end{lemma}

\begin{proof}Let $p\gr p(\mu)$. By the continuity of measure there exists $R(p)>0$ such that
$\mu(\{x\in\mathbb{R}^n:|x|\ls R(p)\})>1-e^{-p}$. Then, for any $x\in\mathbb{R}^n$
with $|x|>R(p)$ we have that $\varphi_{\mu}(x)<e^{-p}$. This shows
that $T_p(\mu)$ is bounded. To see that $T_p(\mu)$ is convex, note that if $x,y\in T_p(\mu)$ then for
any $z\in [x,y]$ and any $H^+\in {\cal H}(z)$ we have that either $x$ or $y$ belongs to $H^+$, and hence $\mu(H^+)\gr \min\{\varphi_{\mu}(x),\varphi_{\mu}(y)\}\gr e^{-p}$, therefore $\varphi_{\mu}(z)\gr e^{-p}$, which implies that $z\in T_p(\mu)$.
Finally, $\varphi_{\mu}$ is upper-semicontinuous (see \cite[Lemma~6.1]{Donoho-Gasko-1992}), therefore
$T_p(\mu)$ is closed.

The fact that the value $p(\mu)$ is attained is proved in \cite[Proposition~7]{RR-1999}. It is also
clear from the definition that $\{T_p(\mu)\}_{p>0}$ is an increasing family of sets. Therefore, $T_p(\mu)$
is a non-empty compact convex set for all $p\gr p(\mu)$.
\end{proof}

Every point $x$ that satisfies $\varphi_{\mu}(x)=e^{-p(\mu)}$ is called a center point for $\mu$. We note that if $\mu$ is symmetric, then $0$ is a center point.

\smallskip

Our first main goal in this section is to prove Theorem~\ref{th:deterministic}, which extends Theorem~\ref{th:HLO} to
the case of an arbitrary Borel probability measure $\mu$ on $\mathbb{R}^n$. The proof makes use of the $\varepsilon$-net theorem.
In order to formulate the latter, we need to introduce a number of notions. Let ${\cal F}$ be a family of subsets of a set $\Omega$.
The Vapnik-Chervonenkis dimension (VC-dimension) of ${\cal F}$ is the maximal cardinality of a finite set $A\subset \Omega$
which is shattered by ${\cal F}$, i.e.
$$\{V\cap A:V\in {\cal F}\}=2^A.$$
We also say that a set $B\subset \Omega$ is a transversal of ${\cal F}$ if $B\cap V\neq\varnothing $ for all $V\in {\cal F}$.

We shall use a strong version of the $\varepsilon$-net theorem that was proved by Koml\'{o}s, Pach and Woeginger in~\cite[Theorem~3.2]{Komlos-Pach-Woeginger}.

\begin{theorem}\label{th:epsilon-net}Let $\varepsilon>0$ and ${\cal F}$ be a family of subsets of a probability space $(\Omega,\mu )$ such that ${\cal F}$ has VC-dimension at most $d$ and $\mu (V)\gr\varepsilon $ for all $V\in {\cal F}$. Then, for any positive integers $M>N$
we have that $N$ independent random points $X_1,\ldots ,X_N$ distributed according to $\mu $ form a transversal of ${\cal F}$ with probability greater than
\begin{equation}\label{eq:almost-tight}1-2\left(\sum_{i=0}^d\binom{M}{i}\right)\left(1-\frac{N}{M}\right)^{(M-N)\varepsilon -1}.\end{equation}
\end{theorem}

Choosing $\displaystyle M=\left\lfloor \frac{\varepsilon N^2}{d}\right\rfloor$ and using the fact that
$\displaystyle \sum_{i=0}^d\binom{M}{i}\ls\left(\frac{eM}{d}\right)^d$, from \eqref{eq:almost-tight} we get the following
corollary of Theorem~\ref{th:epsilon-net}, which appears in \cite[Lemma~3.2]{Naszodi}.

\begin{lemma}\label{lem:naszodi}Let $0<\varepsilon <1/e$, $\gamma >2$ and $d\in {\mathbb N}$. Consider
a family ${\cal F}$ of subsets of a probability space $(\Omega,\mu )$ such that ${\cal F}$ has VC-dimension at most $d$ and $\mu (V)\gr\varepsilon $ for all $V\in {\cal F}$. Then, for any $N\gr\gamma\,\frac{d}{\varepsilon }\ln \frac{1}{\varepsilon }$
we have that $N$ independent random points $X_1,\ldots ,X_N$ distributed according to $\mu $ form a transversal of ${\cal F}$ with probability greater than $1-4\big(11\gamma^2\varepsilon^{\gamma -2}\big)^d$.
\end{lemma}

Besides Lemma~\ref{lem:naszodi} we use the next lemma which provides the exact value of the VC-dimension of the
family of all closed half-spaces that support a compact convex set in $\mathbb{R}^n$.

\begin{lemma}\label{lem:exact-vc}Let $C$ be a non-empty compact convex set in $\mathbb{R}^n$.
Consider the family $\mathcal{H}(C)$ of all half-spaces $H^+$ for which $H$ is a supporting hyperplane
of $C$ and $C\subseteq H^-$. Then, the ${\rm VC}$-dimension of $\mathcal{H}(C)$ is equal to $n$.
\end{lemma}

\begin{proof}First we show that ${\rm VC}(\mathcal{H}(C))\gr n$. We may assume that $0\in C$ and set
$m=\max\{h_C(u):u\in S^{n-1}\}$. Note that $m\gr 0$. Then, we define $r=(1+m)\sqrt{n}$ and consider
the vectors $y_i=re_i$, $1\ls i\ls n$, where $\{e_1,\ldots ,e_n\}$ is the standard orthonormal
basis of $\mathbb{R}^n$. For every $\sigma=(\sigma_1,\ldots ,\sigma_n)\in\{-1,1\}^n$ we consider the vector
$u_{\sigma}=\frac{1}{\sqrt{n}}\sigma\in S^{n-1}$. Then, for every $1\ls i\ls n$ we have
$$\langle y_i,u_{\sigma}\rangle =\frac{1}{\sqrt{n}}r\sigma_i=(1+m)\sigma_i,$$
and this implies that
$$\langle y_i,u_{\sigma}\rangle >h_C(u_{\sigma})\;\;\hbox{if}\;\sigma_i=1$$
while
$$\langle y_i,u_{\sigma}\rangle <0\ls h_C(u_{\sigma})\;\;\hbox{if}\;\sigma_i=-1.$$
It follows that $\mathcal{H}(C)$ shatters the set $\{y_1,\ldots ,y_n\}$.

It remains to show that ${\rm VC}(\mathcal{H}(C))\ls n$. We assume that there exist distinct
$x_1,\ldots ,x_{n+1}\in\mathbb{R}^n$ such that $\mathcal{H}(C)$ shatters the set $A=\{x_1,\ldots ,x_{n+1}\}$.
Consider an arbitrary point $x\in C$. Then, $x_1,\ldots ,x_{n+1},x$ are affinely dependent, and hence there
exist $t_1,\ldots ,t_{n+1},t\in\mathbb{R}$, not all of them equal to $0$, such that
$$t+\sum_{i=1}^{n+1}t_i=0\quad\hbox{and}\quad tx+\sum_{i=1}^{n+1}t_ix_i=0.$$
Without loss of generality we assume that $t\gr 0$. By the definition, the elements of $\mathcal{H}(C)$
are the half-spaces of the form
$$H^+(h_C(u),u)=\{y\in\mathbb{R}^n:\langle y,u\rangle \gr h_C(u)\},\quad u\in S^{n-1}.$$
For any $1\ls i\ls n+1$ and $u\in S^{n-1}$ we define
$$g_i(u)=\langle x_i,u\rangle -h_C(u).$$
Therefore, $x_i\in H^+(h_C(u),u)$ if and only if $g_i(u)\gr 0$. Note that, for any $u\in S^{n-1}$,
\begin{equation}\label{eq:exact-1}t(\langle x,u\rangle -h_C(u))+\sum_{i=1}^{n+1}t_ig_i(u)
=\Big\langle tx+\sum_{i=1}^{n+1}t_ix_i,u\Big\rangle -\left(t+\sum_{i=1}^{n+1}t_i\right)\,h_C(u)=0.\end{equation}
Assume that there exists $1\ls s\ls n+1$ such that $t_s>0$. We define $I=\{1\ls i\ls n+1:t_i>0\}$
and $J=\{1,\ldots ,n+1\}\setminus I$. Since $\mathcal{H}(C)$ shatters $A$, we
may find $u\in S^{n-1}$ such that $g_i(u)<0$ for all $i\in I$ and $g_i(u)\gr 0$ for all $i\in J$.
Then, combining the fact that $t(\langle x,u\rangle -h_C(u))\ls 0$ (because $t\gr 0$ and $x\in C$) with the fact
that $t_sg_s(u)<0$ and $t_ig_i(u)\ls 0$ for all other $1\ls i\ls n+1$, we see that
$$t(\langle x,u\rangle -h_C(u))+\sum_{i=1}^{n+1}t_ig_i(u)<0,$$
which is a contradiction by \eqref{eq:exact-1}.

This means that $t_i\ls 0$ for all $1\ls i\ls n+1$, and hence $t=-\sum_{i=1}^{n+1}t_i>0$. Then,
$$x=\sum_{i=1}^{n+1}\left(-\frac{t_i}{t}\right)x_i$$
is a convex combination of $x_1,\ldots ,x_{n+1}$. Since $x\in C$ was arbitrary, we conclude that
$$C\subseteq {\rm conv}(A).$$
This leads to a contradiction. Using again the fact that $\mathcal{H}(C)$ shatters $A$, we can find $H^+\in\mathcal{H}(C)$
such that $A\cap H^+=\varnothing$. But this implies that $C\cap H^+=\varnothing$, which cannot be true
because the boundary of $H^+$ is a supporting hyperplane of $C$.
\end{proof}

\begin{proof}[Proof of Theorem~$\ref{th:deterministic}$]Let $p\gr p(\mu)$ and consider the convex body $T_p(\mu)$.
Let ${\cal F}_p(\mu)$ be the family of all closed half-spaces $H^+(\alpha,u)$
with the property that $H(\alpha,u)$ is a supporting hyperplane of $T_p(\mu)$ and $T_p(\mu)\subseteq H^-(\alpha,u)$.
From Lemma~\ref{lem:exact-vc} we know that the VC-dimension of ${\cal F}_p(\mu)$ is equal to $n$.

Note that $\mu(H^+(\alpha,u))\gr e^{-p}$ for all $H^+(\alpha,u)\in {\cal F}_p(\mu)$. Let $N\gr c_1(\beta)n$, where
$c_1(\beta)$ is a positive constant, depending only on $\beta$, to be determined.
We set $p=\beta\ln\left(\frac{N}{n}\right)$ and define $\gamma$ by the equation
$p\gamma =\left(\frac{N}{n}\right)^{1-\beta}$. Note that with this choice of $p$ and $\gamma$ we have that
$$\gamma\,\frac{n}{e^{-p}}\ln \frac{1}{e^{-p}}=p\gamma \,e^pn=\left(\frac{N}{n}\right)^{1-\beta}\left(\frac{N}{n}\right)^{\beta}n=N.$$
We need to guarantee that
$$\gamma=\left(\frac{N}{n}\right)^{1-\beta}\big/\,\beta\ln\left(\frac{N}{n}\right)>2,$$
which is certainly true if $N\gr c_1(\beta)n$ for a large enough positive constant $c_1(\beta)$ depending only
on $\beta$, because $\lim\limits_{y\to +\infty}y^{1-\beta}/(\beta\ln y)=+\infty $.
We first claim that
\begin{equation}\label{eq:epsilon-1}44\gamma^2e^{2p}<e^{p\gamma/2}.\end{equation}
By the definition of $p$ and $\gamma$, \eqref{eq:epsilon-1} is equivalent to the inequality
$$44\left(\frac{N}{n}\right)^{2-2\beta}\frac{1}{\beta^2\ln^2\left(\frac{N}{n}\right)}\left(\frac{N}{n}\right)^{2\beta}<
\exp\left(\frac{1}{2}\left(\frac{N}{n}\right)^{1-\beta}\right),$$
and setting $y=\left(\frac{N}{n}\right)^{1-\beta}$ we see that the last inequality takes the form
\begin{equation}\label{eq:epsilon-2}\frac{44(1-\beta)^2}{\beta^2}\,
\frac{y^{\frac{2}{1-\beta}}}{(\ln y)^2}<e^{y/2},\end{equation}
which is clearly true if $y\gr c(\beta)$, or equivalently, $N\gr c_2(\beta)n$ for a suitable constant
$c_2(\beta)>0$. Now, from \eqref{eq:epsilon-1} we get
$$4\big(11\gamma^2e^{-p(\gamma -2)}\big)^{n} \ls (44\gamma^2e^{2p}e^{-p\gamma})^{n}
\ls e^{-p\gamma n/2}$$
and
\begin{equation}\label{eq:epsilon-3}e^{-p\gamma n/2}\ls e^{-c_1N^{1-\beta} n^{\beta}}\end{equation}
if $\frac{1}{2} p\gamma=\frac{1}{2}\left(\frac{N}{n}\right)^{1-\beta}\gr c_1\left(\frac{N}{n}\right)^{1-\beta}$,
which is certainly true if $0<c_1\ls\frac{1}{2}$.

Then, Lemma~\ref{lem:naszodi} shows that if $N\gr c_1(\beta)n$ and $X_1,\ldots ,X_N$ are independent random points
distributed according to $\mu $ then the set $B=\{X_1,\ldots ,X_N\}$ forms a transversal of ${\cal F}_{\beta\ln\left(\frac{N}{n}\right)}(\mu)$
with probability greater than $1-\exp(-\tfrac{1}{2}N^{1-\beta}n^{\beta})$.

Finally, note  that if $B=\{X_1,\ldots ,X_N\}$ is a transversal of ${\cal F}_{\beta\ln\left(\frac{N}{n}\right)}(\mu)$ then
$$T_{\beta\ln\left(\frac{N}{n}\right)}(\mu)\subseteq {\rm conv}(B)=K_N.$$
To see this, assume that there exists $x\in {\rm bd}(T_{\beta\ln\left(\frac{N}{n}\right)}(\mu))\setminus {\rm conv}(B)$.
Then, we may find a hyperplane $H(\beta ,u)$ such that ${\rm conv}(B)\subseteq H^-(\beta,u)$ and $\langle x,u\rangle =\alpha >\beta$. We set $\alpha_0= \max \{ \langle u, z \rangle : z \in T_{\beta\ln\left(\frac{N}{n}\right)}(\mu) \}$, which is attained at some point
$y \in {\rm bd}(T_{\beta\ln\left(\frac{N}{n}\right)}(\mu))$. It is clear that $H^+(\alpha_0,u) \in {\cal F}_{\beta\ln\left(\frac{N}{n}\right)}(\mu)$ and $H^+(\alpha_0,u) \cap B =\emptyset$, since $\alpha_0\gr \alpha >\beta$. This leads to a contradiction.
\end{proof}

Next, we describe the relation between $U_p(\mu)$ and $T_p(\mu)$.
A similar result, when $\mu$ is assumed symmetric, appears in \cite{HLO-2023}. In what follows, for every $p>0$ we also define
$$V_p(\mu)=\{y\in \mathbb{R}^n:\;\mu(\{x\in\mathbb{R}^n:\langle x,y\rangle\gr 1\}) < \exp(-p)\}$$
and
$$S_p(\mu)=\{x\in\mathbb{R}^n:\varphi_{\mu}(x)>e^{-p}\}.$$

Note that $U_p(\mu)$ and $V_p(\mu)$ are star-shaped at the  origin for every $p>0$.

\begin{proposition}\label{prop:U-T}Let $\mu$ be a Borel probability measure on $\mathbb{R}^n$. If $0$ is a center point for $\mu$ then for every $p\gr p(\mu)$ we have that
\begin{equation}\label{eq:U-T}S_p(\mu)\subseteq (U_p(\mu))^{\circ}\subseteq T_p(\mu).\end{equation}
\end{proposition}

\begin{proof}Let $p\gr p(\mu)$. It is clear that $V_p(\mu)\subseteq U_p(\mu)$, and hence
$(U_p(\mu))^{\circ}\subseteq (V_p(\mu))^{\circ}$. We shall show that
\begin{equation}\label{eq:V-T}T_p(\mu)=(V_p(\mu))^{\circ},\end{equation}
which proves the right-hand side inclusion of \eqref{eq:U-T}.

We start with the observation that $x\in (V_p(\mu))^{\circ}$ if and only if for every $y\neq 0$ we have the implication
\begin{equation}\label{eq:U-T-1}\mu(\{z:\langle z,y\rangle\gr 1\})<e^{-p}\;\;\Longrightarrow\;\;\langle x,y\rangle\ls 1.\end{equation}
If we write $y=\frac{1}{r}\xi$ where $r>0$ and $\xi\in S^{n-1}$ then we see that \eqref{eq:U-T-1} is equivalent
to the following statement: for every $r>0$ and $\xi\in S^{n-1}$,
\begin{equation}\label{eq:U-T-2}\mu(\{z:\langle z,\xi\rangle\gr r\})<e^{-p}\;\;\Longrightarrow\;\;\langle x,\xi\rangle\ls r.\end{equation}
Since $0$ is a center point for $\mu$, we also know that, for every $r<0$ and $\xi\in S^{n-1}$,
$$\mu(\{z:\langle z,\xi\rangle\gr r\})\gr \mu(\{z:\langle z,\xi\rangle\gr 0\})\gr \varphi_{\mu}(0)
= e^{-p(\mu)}\gr e^{-p},$$
therefore the implication \eqref{eq:U-T-2} continues to hold. In other words, $x\in (V_p(\mu))^{\circ}$ if and only if
for every $r\in\mathbb{R}$ and $\xi\in S^{n-1}$ we have that
\begin{equation}\label{eq:U-T-3}\mu(\{z:\langle z,\xi\rangle\gr r\})<e^{-p}\;\;\Longrightarrow\;\;\langle x,\xi\rangle\ls r.\end{equation}
This is in turn equivalent to the next statement: for any $\xi\in S^{n-1}$,
\begin{equation}\label{eq:U-T-4}\mu(\{z:\langle z,\xi\rangle \gr \langle x,\xi\rangle \})\gr e^{-p}.\end{equation}
To see this, assume that there exists $\xi\in S^{n-1}$ such that $\mu(\{z:\langle z,\xi\rangle \gr \langle x,\xi\rangle \})< e^{-p}$.
Then, we may find $\delta>0$ so that $\mu(\{z:\langle z,\xi\rangle \gr \langle x,\xi\rangle -\delta \})<e^{-p}$,
and applying \eqref{eq:U-T-3} with $r=\langle x,\xi\rangle -\delta $ we get $\langle x,\xi\rangle \leq r$, which implies
that $\delta\ls 0$, a contradiction. Now, we readily see that \eqref{eq:U-T-4} is equivalent to
$$\varphi_{\mu}(x)=\inf_{\xi\in S^{n-1}}\mu(\{z:\langle z,\xi\rangle \gr \langle x,\xi\rangle \})\gr e^{-p},$$
and hence to $x\in T_p(\mu)$. So, we have proved \eqref{eq:V-T}.

For the left-hand side inclusion on \eqref{eq:U-T} note that if $p(\mu)\ls q<p$ then $U_p(\mu)\subseteq V_q(\mu)$,
and hence $(U_p(\mu))^{\circ}\supseteq (V_q(\mu))^{\circ}=T_q(\mu)$. It follows that
$$S_p(\mu)=\bigcup_{p(\mu)\ls q<p}T_q(\mu)\subseteq (U_p(\mu))^{\circ}$$
and the proof of the lemma is now complete.\end{proof}

\begin{remark}\label{rem:equality}\rm  If we assume that $\mu$ has a density $f_{\mu}$ and $D=\{f_{\mu}>0\}$ coincides with an open set up to a Borel null set then we can check that the sets $S_p(\mu)$ and $T_p(\mu)$ in Proposition~\ref{prop:U-T} have the same measure.
This follows from the fact that for every $t>0$ the set $A_t:=\{x\in D:\varphi_{\mu}(x)=t\}$ has  $\vol_n(A_t)=0$ and thus $\mu(A_t)=0$. To see this, assume that
$\vol_n(A_t)>0$. By Lebesgue's differentiation theorem, there exists a density point $x_0$ of $A_t$, and we
may find $r>0$ such that $\vol_n(A_t\cap B(x_0,r))\gr\frac{2}{3}\vol_n(B(x_0,r))$.
It is proved in \cite{Laketa-Nagy-2022} that $\varphi_{\mu}(x_0)$ is attained for some half-space, i.e.
there exists $\xi\in S^{n-1}$ such that
$$t=\varphi_{\mu}(x_0)=\mu(\{x\in D:\langle x,\xi\rangle \gr \langle x_0,\xi\rangle \}).$$
We set $B^-(x_0,r)=\{y\in B(x_0,r):\langle y,\xi\rangle \ls \langle x_0,\xi\rangle \}$. Then,
$$\vol_n(B^-(x_0,r))=\frac{1}{2}\vol_n(B(x_0,r))<\vol_n(A_t\cap B(x_0,r)),$$
and hence we may find $y\in A_t\cap B(x_0,r)$ such that $\langle y,\xi\rangle >\langle x_0,\xi\rangle $.
Then,
$$t=\varphi_{\mu}(x_0)=\mu(\{x\in D:\langle x,\xi\rangle \gr \langle x_0,\xi\rangle \})
>\mu(\{x\in D:\langle x,\xi\rangle \gr \langle y,\xi\rangle \}) \gr \varphi_{\mu} (y)=t,$$
which is a contradiction.
\end{remark}

We close this section with a proof of Theorem~\ref{th:N-varphi}.

\begin{proof}[Proof of Theorem~$\ref{th:N-varphi}$]
Let $x\in\mathbb{R}^n$ with $\varphi:=\varphi_{\mu}(x)>0$. We apply Theorem~\ref{th:epsilon-net} for the
family ${\cal F}=\mathcal{H}(x)$ of all closed half-spaces $H$ with $x\in H$. Lemma~\ref{lem:exact-vc}
shows that $\mathcal{H}(x)$ has VC-dimension equal to $n$, and $\mu (H)\gr\varphi $ for all $H\in \mathcal{H}(x)$. Then, if we set
$M=N+\frac{N^2}{n}>N$ we have that $N$ independent random points $X_1,\ldots ,X_N$ distributed according to $\mu $ form a transversal of $\mathcal{H}(x)$ with probability greater than $1-2p(n,N,\varphi)$, where
\begin{equation}\label{eq:almost-tight-2}p(n,N,\varphi)=\left(e\frac{nN+N^2}{n^2}\right)^{n}
\left(1-\frac{nN}{nN+N^2}\right)^{\frac{\varphi N^2}{n}}.\end{equation}
Set $y=1/\varphi $ and $N=n\,ay(\ln y)$, where $a>1$ is a constant to be chosen. If $y\gr e$, then
\eqref{eq:almost-tight-2} takes the simpler form
\begin{align*}p(n,N,\varphi)&=\left[eay(\ln y)(1+ay(\ln y))\left(1-\frac{1}{1+ay\ln y}\right)^{a^2y(\ln y)^2}\right]^{n}\\
&\ls \left[2a^2ey^2(\ln y)^2\exp\left(-a(\ln y)/2\right)\right]^{n}.\end{align*}
Now, it is clear that if we choose $a=6$ and $1/\varphi =y\gr c_1$ where $c_1>1$ is an absolute constant, then
$p(n,N,\varphi )\ls 2^{-n}$, and hence we get that with probability greater than $1/2$ the random
vectors $X_1,\ldots ,X_N$ form a transversal of $\mathcal{H}(x)$, which easily implies that $x\in {\rm conv}\{X_1,\ldots ,X_N\}$.
Since our choice of $N$ gives $N=6ny(\ln y)=\frac{6n}{\varphi}\,\ln (1/\varphi)$, by the definition of $N_{\mu}(x)$
we see that
$$N_{\mu}(x)\ls\frac{6n}{\varphi_{\mu}(x)}\,\ln\left(1/\varphi_{\mu}(x)\right)$$
if $\varphi_{\mu}(x)\ls c_1^{-1}$. The result follows with a simple computation for the case
$\varphi_{\mu}(x)\gr c_1^{-1}$.
\end{proof}

\begin{note*}Let us add here that a reverse inequality can be obtained in a simple way. If $N$ is an integer
that satisfies $\frac{1}{2N}>\varphi_{\mu}(x)$ then there exists $\xi\in S^{n-1}$ such that
$\mu(\{y:\langle y-x,\xi\rangle\ls 0\})<\frac{1}{2N}$. It follows that
$$\mathbb{P}(x\in {\rm conv}\{X_1,\ldots ,X_N\})\ls \mathbb{P}\left(\bigcup_{i=1}^N
\{\langle X_i-x,\xi\rangle\ls 0\}\right)\ls N\mathbb{P}(\langle X-x,\xi\rangle\ls 0)<\frac{1}{2}.$$
Therefore, $N_{\mu}(x)\gr\frac{1}{2\varphi_{\mu}(x)}$.

Hayakawa, Lyons and Oberhauser give examples which show that both the upper and the lower bound
in the inequality $\frac{1}{2}\ls N_{\mu}(x)\varphi_{\mu}(x)\ls 3n$ are tight up to absolute constants, even for
small values of $\varphi_{\mu}(x)$ (see \cite[Remark~4]{HLO-2023} and \cite[Example~35]{HLO-2023}, respectively).
\end{note*}

%%%%%%%%%%%%%%%%%%%%%%%%%%%%%%%%%%%%%%%%%%%%%%%%%%%%%%%%%%%%%%%%%%%%%%%%%%%%%%%%%%%%%%%%%%%%%%%%%%%%%%%%%%%%%%%%%%
\section{Regular and strongly regular measures}\label{section-4}
%%%%%%%%%%%%%%%%%%%%%%%%%%%%%%%%%%%%%%%%%%%%%%%%%%%%%%%%%%%%%%%%%%%%%%%%%%%%%%%%%%%%%%%%%%%%%%%%%%%%%%%%%%%%%%%%%%

Let $\mu$ be a Borel probability measure on $\mathbb R^n$.  In order to apply Theorem~\ref{th:deterministic} in concrete situations we need estimates for the size of the bodies $T_p(\mu)$. In this section we compare  the body $T_p(\mu)$
with the nonsymmetric $L_p$-centroid bodies $Z_p^+(\mu)$ of $\mu$, under some regularity assumptions on the measure $\mu$.
The centroid bodies are easier to handle. For example, as we will see in the next section, we can provide general lower
bounds for their volume.

For every $p\gr 1$ we consider the compact convex set $Z_p^+(\mu )$ with support function
\begin{equation*}h_{Z_p^+(\mu )}(y)=\left (\int_{{\mathbb R}^n}\langle
x,y\rangle_+^pd\mu(x)\right )^{1/p},\qquad y\in {\mathbb R}^n,\end{equation*} where $a_+=\max
\{a,0\}$, provided that $h_{Z_p^+(\mu )}$ is bounded on $S^{n-1}$.

If $A \subseteq \mathbb{R}^n$ is Borel measurable with $\vol_n(A)=1$, then for $p \gr 1$ we denote $Z_{p}^+(A) := Z_{p}^+(\mu_A),$ where $\mu_A$ is the uniform measure on $A$.

\begin{claim}\label{claim:Z-interior}Let $\mu$ be a full-dimensional Borel probability measure on $\mathbb R^n$. If $Z_{p}^+(\mu)$ is well defined for some $p \gr 1$, then it is a convex body. Moreover, if $\mu$ is centered, then $0 \in {\rm int}(Z_{p}^+(\mu))$.

\end{claim}

\begin{proof}[Proof of Claim~$\ref{claim:Z-interior}$] Since for every $\xi \in S^{n-1}, \, \mu(\{x:\langle x, \xi \rangle=0\})<1 $, it is easy to verify that
$$ h_{Z_{p}^+(\mu)}(\xi)+h_{Z_{p}^+(\mu)}(-\xi)>0$$
holds true for every $\xi \in S^{n-1}$ and thus $Z_{p}^+(\mu)$ has non empty interior.

Assuming now that $\rm{bar}(\mu)=0$, by the continuity of $h_{Z_{p}^+(\mu)}$, it is enough to show that $h_{Z_{p}^+(\mu)}(\xi)>0$ for every $\xi \in S^{n-1}$. Considering otherwise, if $h_{Z_{p}^+(\mu)}(\xi)=0$ for some $\xi \in S^{n-1}$, then $\mu(\{x: \langle x, \xi \rangle \ls 0\})=1$. So we must have that
$$\int_{\mathbb{R}^n} \langle x, \xi \rangle d\mu(x) \ls 0.$$
But the above inequality is an equality, since $\rm{bar}(\mu)=0$. This means that $\mu(\{x:\langle x,\xi \rangle=0\})=1$, which is a contradiction.
\end{proof}

\begin{definition}\label{def:Z}\rm Let $\mu$ be a Borel probability measure on $\mathbb{R}^n$. We say that
$\mu$ is {\it $\alpha $-regular} if $Z_1^+(\mu)$ is a compact convex set and
$$\|\langle \cdot,y\rangle_+\|_{L^{2p}(\mu)}
\ls 2\alpha \,\|\langle \cdot ,y\rangle_+\|_{L^{p}(\mu)}$$
for every $y\in\mathbb{R}^n$ and any $p\gr 1$. Equivalently, if $Z_{2p}^+(\mu)\subseteq 2\alpha \,Z_p^+(\mu)$
for every $p\gr 1$. We also say that $\mu$ is {\it $\alpha $-strongly regular} if $Z_1^+(\mu)$ is a convex body and
$$\|\langle \cdot ,y\rangle_+\|_{L^{q}(\mu)}
\ls \frac{\alpha q}{p}\,\|\langle \cdot ,y\rangle_+\|_{L^{p}(\mu)}$$
for every $y\in\mathbb{R}^n$ and any $q>p\gr 1$. Equivalently, if $Z_q^+(\mu)\subseteq \frac{\alpha q}{p}\,Z_p^+(\mu)$
for every $q>p\gr 1$. It is clear that every $\alpha $-strongly regular
Borel probability measure is $\alpha $-regular.

Every centered log-concave probability measure is $C$-strongly regular,
and hence $C$-regular, where $C>0$ is an absolute constant. Indeed, one can check that if $1\ls p<q$ then
\begin{equation}\label{eq:regularity}\left(\frac{4}{e}\right)^{\frac{1}{p}-\frac{1}{q}}Z_p^+(\mu )\subseteq Z_q^+(\mu )
\subseteq c_1\left(\frac{4(e-1)}{e}\right)^{\frac{1}{p}-\frac{1}{q}}\frac{q}{p}Z_p^+(\mu ).\end{equation}
For a proof see \cite{Guedon-EMilman-2011}.
\end{definition}

\begin{proposition}\label{prop:measure-1}Let $\mu$ be an $\alpha$-regular Borel probability measure on $\mathbb{R}^n$.
Then, for every $p\gr 1$ we have that
$$Z_p^+(\mu)\subseteq 2T_{g(p)}(\mu)$$
where $g(p)=\max\{2\ln(2e\alpha )p,\ln(1/\varphi_{\mu}(0))\}$.
\end{proposition}

\begin{proof}Let $x\in \tfrac{1}{2}Z_p^+(\mu )$. For any $\xi\in S^{n-1}$ with $\langle x,\xi\rangle\gr 0$
we have $\langle x,\xi\rangle\ls \tfrac{1}{2}h_{Z_p^+(\mu )}(\xi )$ and hence
$$\mu (\{z\in {\mathbb R}^n:\langle z,\xi\rangle \gr\langle x,\xi\})\gr\mu(\{z\in {\mathbb R}^n:\langle z,\xi\rangle \gr \tfrac{1}{2} h_{Z_p^+(\mu )}(\xi)\}).$$
We apply the Paley-Zygmund inequality
\begin{equation*}\mu (\{z:h(z)\gr 2^{-p}{\mathbb E}_{\mu }(h)\})\gr
(1-2^{-p})^2\frac{\big({\mathbb E}_{\mu }(h)\big)^2}{{\mathbb E}_{\mu }(h^2)}\gr
\frac{1}{4}\,\frac{\big({\mathbb E}_{\mu }(h)\big)^2}{{\mathbb E}_{\mu }(h^2)}
\end{equation*} for the function $h(z)=\langle z,\xi\rangle_+^p$. Since $\mu$ is $\alpha $-regular, we see that
\begin{equation*}{\mathbb E}_{\mu }(h^2)\ls (2\alpha )^{2p}\,\big({\mathbb E}_{\mu }(h)\big)^2.\end{equation*}
Therefore,
$$\mu (\{z\in {\mathbb R}^n:\langle z,\xi\rangle \gr\langle x,\xi\})\gr
\frac{1}{4}e^{-2\ln (2\alpha )p}\gr e^{-2\ln (2\alpha )p-2}\gr e^{-2\ln(2e\alpha)p}.$$
On the other hand, if $\langle x,\xi\rangle <0$ then it is clear that
$$\mu (\{z\in {\mathbb R}^n:\langle z,\xi\rangle \gr\langle x,\xi\})\gr \varphi_{\mu}(0)= e^{-\ln (1/\varphi_{\mu}(0))}.$$
This shows that
$$\varphi_{\mu}(x)=\inf_{\xi\in S^{n-1}}\mu (\{z\in {\mathbb R}^n:\langle z,\xi\rangle \gr\langle x,\xi\})
\gr \exp\Big(-\max\{2\ln(2e\alpha )p,\ln(1/\varphi_{\mu}(0))\}\Big)$$
and the proposition follows. \end{proof}

For an $\alpha$-strongly regular measure $\mu$ we can also establish the equivalence of the family $\{T_p(\mu)\}_{p\gr p(\mu)}$
with the family $\{B_p(\mu)\}_{p>0}$ of the level sets of the Cram\'{e}r transform of $\mu$.
Recall that if $\mu$ is a Borel probability measure on $\mathbb R^n$ then the log-Laplace transform of $\mu$ is defined by
\begin{equation*}\Lambda_{\mu }(\xi )=\ln\left(\int_{{\mathbb R}^n}e^{\langle\xi
,z\rangle }d\mu(z)\right)\end{equation*}
and the Cram\'{e}r transform $\Lambda_{\mu}^{\ast }$ of $\mu $ is the Legendre transform of $\Lambda_{\mu}$, defined by
\begin{equation*}\Lambda_{\mu }^{\ast }(x)= \sup_{\xi\in {\mathbb R}^n} \left\{ \langle x, \xi\rangle - \Lambda_{\mu }(\xi )\right\}.\end{equation*}
Note that $\Lambda_{\mu}^{\ast}$ is a non-negative convex function. For any $p>0$ we define
$$B_p(\mu)=\{x\in\mathbb{R}^n:\Lambda_{\mu}^{\ast}(x)\ls p\}.$$
From the inequality $\varphi_{\mu } (x)\ls\exp (-\Lambda_{\mu }^{\ast }(x))$ which is a direct
consequence of the definitions (see e.g. \cite[Lemma~3.1]{BGP-threshold}) we immediately get the next
lemma.

\begin{lemma}\label{lem:standard}Let $\mu$ be a Borel probability measure on $\mathbb R^n$.
For every $p>0$ we have that $T_p(\mu)\subseteq B_p(\mu)$.\end{lemma}

If $\mu$ is $\alpha $-strongly regular then the next proposition establishes a reverse inclusion between the bodies
$B_p(\mu )$ and $Z_p^+(\mu )$ (a variant of this result appears in \cite[Proposition~3.5]{Latala-Wojtaszczyk-2008}).

\begin{proposition}\label{prop:B<Z}Let $\mu $ be an $\alpha$-strongly regular probability measure on $\mathbb{R}^n$.
Then, for every $p\gr 1$ and any $\delta \in (0,1]$ we have that
\begin{equation*}B_p(\mu) \subseteq (1+\delta)Z_{c_1\alpha p/\delta}^+(\mu)\end{equation*}
where $c_1>0$ is an absolute constant.
\end{proposition}

\begin{proof}Let $\delta \in (0,1]$, $p\gr 1$ and $q\gr 1$ that will be suitably chosen
(depending on $p$). For the proof it is convenient to define the set
\begin{equation*}W^+_p(\mu) := \left\{y\in {\mathbb R}^n : \int_{{\mathbb R}^n} \langle y, x\rangle_+^pd\mu(x)\ls 1\right\}\end{equation*}
for every $p\gr 1$. Note that
\begin{equation*}Z^+_p(\mu) = (W^+_p(\mu))^{\circ}.\end{equation*}
If $y\in W^+_q(\mu)$ then H\"{o}lder's inequality shows that
$\|\langle y,\cdot\rangle_+\|_k\ls \|\langle y,\cdot\rangle_+\|_q\ls 1$ for all $k\ls q$, and the
strong $\alpha $-regularity of $\mu$ implies that
$$\|\langle y,\cdot\rangle_+\|_k\ls \frac{\alpha k}{q}\|\langle y,\cdot\rangle_+\|_q\ls \frac{\alpha k}{q}$$
for all $k>q$. Since $\frac{k}{(k!)^{1/k}}\to e$, we may choose a constant $\gamma\approx 1/\alpha $ so
that $\frac{\alpha \gamma k}{(k!)^{1/k}}\ls\frac{1}{2}$ for all $k\gr 1$.  It follows that
\begin{align*}
\int_{{\mathbb R}^n} e^{\langle \gamma qy,x\rangle_+}d\mu(x)
&=\sum_{k=0}^{\infty}\frac{1}{k!}\int_{{\mathbb R}^n}\langle \gamma qy, x\rangle_+^kd\mu(x)
\ls \sum_{k\ls q}\frac{(\gamma q)^k}{k!}+\sum_{k>q}\frac{(\gamma q)^k}{k!}\left (\frac{\alpha k}{q}\right )^k\\
&\ls e^{\gamma q}+\sum_{k>q}\frac{1}{2^k}\ls e^{\gamma q}+1\ls e^{\gamma q+1}.\end{align*}
Therefore, for any $y\in W^+_q(\mu)$ we get $\Lambda_{\mu}(\gamma qy)\ls\gamma q+1$.

Now, let $x\notin (1+\delta)Z^+_q(\mu)$. We can find $y\in W^+_q(\mu)$ such that $\langle x,y\rangle> 1+\delta $ and then
\begin{equation*}\Lambda^{\ast }_{\mu }(x) \gr \langle x,\gamma qy\rangle
-\Lambda_{\mu }(\gamma qy)> (1+\delta)\gamma q -\gamma q-1=\delta\gamma q-1 \gr p\end{equation*}
if we assume that $q\gr \frac{2p}{\gamma\delta}$. Therefore, $x\notin
B_p(\mu )$. This shows that $B_p(\mu)\subseteq (1+\delta)Z^+_{c_1\alpha p/\delta} (\mu)$, where $c_1>0$
is an absolute constant. \end{proof}

Combining Proposition~\ref{prop:measure-1} and Proposition~\ref{prop:B<Z} we see that if $\mu$ is
$\alpha $-strongly regular then the bodies $Z_p^+(\mu)$, $T_p(\mu)$
and $B_p(\mu)$ are equivalent up to constants that do not depend on $p$.

\begin{theorem}\label{th:summary-regular}
Let $\mu$ be an $\alpha$-regular Borel probability measure on $\mathbb{R}^n$. If $0$ is a center point for
$\mu$ then for every $p\gr \frac{1}{2\ln (2e\alpha )}\ln (n+1)$ we have that
\begin{equation}\label{eq:summary}Z_p^+(\mu)\subseteq 2T_{2\ln(2e\alpha )p}(\mu).\end{equation}
Moreover, if $\mu$ is $\alpha $-strongly regular then
$T_p(\mu)\subseteq B_p(\mu) \subseteq 2Z_{c_1\alpha p}^+(\mu)$ for every $p\gr 1$, where $c_1>0$ is an
absolute constant.
\end{theorem}

\begin{proof}The inclusion of \eqref{eq:summary} follows by the assumption that $0$ is a center point for $\mu$,
and hence $\varphi_{\mu}(0)\gr\frac{1}{n+1}$, which implies that $\ln(1/\varphi_{\mu}(0))\ls\ln (n+1)$.
Therefore, if $p\gr \frac{1}{2\ln (2e\alpha )}\ln (n+1)$ then we have that
$g(p)=2\ln(2e\alpha )p$ in Proposition~\ref{prop:measure-1}.
\end{proof}

Next, we restrict our attention to the class of log-concave probability measures. If $\mu$ is a centered log-concave
probability measure on $\mathbb{R}^n$ then Gr\"{u}nbaum's lemma (see \cite[Lemma~2.2.6]{BGVV-book}) shows that
$\mu(\{z:\langle z,\xi \rangle\gr 0\})\gr 1/e$ for every $\xi\in S^{n-1}$, and hence $\ln(1/\varphi_{\mu}(0))\ls 1$.
We also know that $\mu$ is $C$-regular, by \eqref{eq:regularity}.
Therefore, Proposition~\ref{prop:measure-1} implies the following: For every $p\gr 1$ we have that
\begin{equation}\label{eq:log-concave-1}Z_p^+(\mu)\subseteq 2T_{cp}(\mu)\end{equation}
where $c>0$ is an absolute constant.

We can also use an alternative approach that gives a more precise version of Theorem~\ref{th:summary-regular}.
The next result, which is essentially due to \L ata\l a and Wojtaszczyk (see \cite[Proposition~3.2]{Latala-Wojtaszczyk-2008}),
provides a direct inclusion relation between the bodies $Z_p^+(\mu)$ and $B_p(\mu)$:
If $\mu$ is a centered probability measure on $\mathbb{R}^n$ then, for every $q\gr p\gr p_0$, where
$p_0$ is an absolute constant, we have that
\begin{equation}\label{eq:log-concave-2}Z_p^+(\mu) \subseteq \left(1+\frac{2\ln q}{q}\right)B_q(\mu).\end{equation}
A proof of this particular statement appears in \cite[Proposition~2.5]{Giannopoulos-Tziotziou-2025}. Note that the symmetry assumption
on $\mu$ (which appears in the work of \L ata\l a and Wojtaszczyk) is not required.
If we make the additional assumption that $\mu$ is log-concave, then we can also compare $B_p(\mu)$
with $T_p(\mu)$. The next result appears in \cite[Proposition~2.7]{Giannopoulos-Tziotziou-2025}:
There exists $p_0\gr 1$ such that, for every centered log-concave
probability measure $\mu$ on $\mathbb{R}^n$ and any $p\gr p_0$,
\begin{equation}\label{eq:log-concave-3}T_p(\mu)\subseteq B_p(\mu)\subseteq T_{p+3\ln p}(\mu),\end{equation}
where $p_0\gr 1$ is an absolute constant. The proof of this fact is based on a theorem of Brazitikos and Chasapis
from \cite{Brazitikos-Chasapis-2024}: If $\mu$ is log-concave then, for every $x\in {\rm supp}(\mu)$
and any $\varepsilon\in (0,1)$ we have that
\begin{equation}\label{eq:log-concave-4}\Lambda_{\mu}^{\ast}(x)\gr (1-\varepsilon)\ln\left(\frac{1}{\varphi_{\mu}(x)}\right)
+\ln\left(\frac{\varepsilon}{2^{1-\varepsilon}}\right)
=\ln\left(\frac{\varepsilon}{(2\varphi_{\mu}(x))^{1-\varepsilon}}\right).\end{equation}
Combining these estimates, we get the next proposition.

\begin{proposition}\label{prop:summary-log-concave}Let $\mu$ be a centered log-concave probability measure on $\mathbb{R}^n$.
For every $p\gr p_0$ we have that $Z_p^+(\mu)\subseteq \left(1+\frac{2\ln p}{p}\right)T_{p+3\ln p}(\mu)$
where $p_0>0$ is an absolute constant.
\end{proposition}

Analogous results may be obtained for $s$-concave measures. Since the class of $s$-concave
measures on $\mathbb{R}^n$ is decreasing in $s$, we are interested in the case $s<0$ (if $s\gr 0$
then every $s$-concave measure $\mu$ is log-concave and Proposition~\ref{prop:summary-log-concave}
compares $Z_p^+(\mu)$ with $T_p(\mu)$). It is known (see Bobkov~\cite{Bobkov-2007}) that if $\mu$ is $(-1/\kappa)$-concave
for some $\kappa>0$ then the density $f_{\mu}$ of $\mu$ satisfies $f_{\mu}(x)\ls C/(1+|x|^{n+\kappa})$
for all $x\in\mathbb{R}^n$. Bobkov also showed (see \cite[Theorem~5.2]{Bobkov-2010}) that if $\mu$ is centered and $(-1/\kappa)$-concave for some $\kappa>1$, then
\begin{equation}\label{eq:s-concave} \left(1-\frac{1}{\kappa}\right)^{\kappa} \ls \mu(\{x:\langle x,u\rangle \gr 0\}) \ls 1-\left(1-\frac{1}{\kappa}\right)^{\kappa}\end{equation}
for every $u \in S^{n-1}$.
A consequence of \cite[Corollary~8]{Fradelizi-Guedon-Pajor-2014} along with the estimate $\left(x\mathrm{B}(x,y)\right)^{1/x} \approx \frac{x}{x+y}$ for every $x,y \gr 1$, is that if
$\mu$ is centered and $(-1/\kappa)$-concave for some $\kappa >2$ then for all $1\ls p<q\ls \kappa-1$ we have that
\begin{equation}\label{eq:k-Zp-star}Z_p^+(\mu)\subseteq Z_q^+(\mu)\subseteq \frac{C_1q}{p}Z_p^+(\mu)\end{equation}
where $C_1>0$ is an absolute constant. In particular, if $\kappa \gr 3$ then
\begin{equation}\label{eq:k-Zp-star-2} Z_{2p}^+(\mu) \subseteq 2C_1 Z_{p}^+(\mu)
\end{equation}
for every $1 \ls p \ls (\kappa-1)/2$.

Other interesting cases where one can apply the results of Section~\ref{section-3} appear in \cite{GKKMR-2022}.
If $K_N$ is the random polytope generated by a random vector $X=(\xi_1,\ldots ,\xi_n)$ whose
coordinates are independent copies of a $q$-stable random variable $\xi$, where $1\ls q<2$, then
for any $\beta\in (0,1)$ and any $N\gr c_1(\beta,q)n$ we have that
$$K_N\supseteq c_2(q)\left(\frac{N}{n}\right)^{\beta/q}B_{q^{\prime}}^n$$
with probability greater than $1-2\exp(-c_3N^{1-\beta}n^{\beta})$, where $q^{\prime}$ is the conjugate
exponent of $q$. In the case $q=1$, which corresponds to a Cauchy random variable $\xi$, we have that
$$K_N\supseteq c_3\left(\frac{N}{n}\right)^{\beta}B_{\infty}^n$$
with the same probability, where $B_{\infty}^n=[-1,1]^n$ is the unit cube.
These assertions follow by a direct computation of the ``size" of $T_p(\mu)$ where $\mu$ is
the distribution of $X$, which is performed in \cite{GKKMR-2022}. Note that if $q<2$ then $\xi$ is heavy tailed; in particular, it
does not have a finite second moment. This explains the fact that $K_N$ is a much larger set than
a ``log-concave random polytope".

%%%%%%%%%%%%%%%%%%%%%%%%%%%%%%%%%%%%%%%%%%%%%%%%%%%%%%%%%%%%%%%%%%%%%%%%%%%%%%%%%%%%%%%%%%%%%%%%%%%%%%%%%%%%%%%%%%%%%%%
\section{Centroid bodies of absolutely continuous measures}
%%%%%%%%%%%%%%%%%%%%%%%%%%%%%%%%%%%%%%%%%%%%%%%%%%%%%%%%%%%%%%%%%%%%%%%%%%%%%%%%%%%%%%%%%%%%%%%%%%%%%%%%%%%%%%%%%%%%%%%

The starting point of this section is an estimate, essentially due to Lutwak, Yang and Zhang \cite{Lutwak-Yang-Zhang-2000},
about the volume of the $L_p$-centroid
bodies of a centered log-concave probability measure $\mu$ on $\mathbb{R}^n$. One has
\begin{equation*}\vol_n(Z_p(\mu))^{1/n}\gr c_1L_{\mu}^{-1}\sqrt{p/n}\gr c_2\sqrt{p/n}\end{equation*}
for every $1\ls p\ls n$, where $c_1,c_2>0$ are absolute constants. Here, $L_{\mu}$ is the
isotropic constant of $\mu$, and the second inequality is a consequence of the recent
affirmative answer to the hyperplane conjecture.

We discuss similar estimates for the volume of the $L_p$-centroid bodies of $\alpha$-regular measures.
To this end, we make use of the family of star-shaped at the origin sets $\{K_p(\mu)\}_{p>0}$
associated with a probability measure $\mu$, introduced by K.~Ball in \cite{Ball-1988}.
Let $f:\mathbb R^n\to [0,\infty)$ be a measurable function such that $f(0) >0$. For any $p>0$ we define
the set $K_p(f)$ as follows:
\begin{equation*}
K_p(f)=\left\{ x\in \mathbb R^n : \int_0^\infty f(rx)r^{p-1} \, dr\gr \frac{f(0)}{p}\right\}.
\end{equation*}
From the definition it follows that the radial function of $K_p(f)$ is given by
\begin{equation}\label{eq:rho}
\rho_{K_p(f)}(x)=\left (\frac{1}{f(0)}\int_0^{\infty}pr^{p-1}f(rx)\,dr\right )^{1/p}\end{equation}
for $x\neq 0$. If $\mu$ is a probability measure on $\mathbb R^n$ which is absolutely continuous with respect
to the Lebesgue measure, with bounded density $f_{\mu}$ and such that $f_{\mu}(0) > 0$, then we define
\begin{equation*}
K_p(\mu)=K_p(f_\mu).
\end{equation*}
We recall a number of known facts about the bodies $K_p(\mu)$, checking that no other assumption on the
density $f_{\mu}$ is needed. First of all,
\begin{align*}
\vol_n(K_{n}(\mu)) &= \int_{K_{n}(\mu)} \mathbf{1}\, dx= n\omega_n \int_{S^{n-1}}
\int_{0}^{\rho_{K_{n}(\mu)}(\xi)} r^{n-1} dr d\sigma(\xi)\\
&= \frac{n\omega_n}{f_{\mu}(0)} \int_{S^{n-1}}
\int_{0}^{\infty}  r^{n-1} f_{\mu}(r\xi) dr d\sigma (\xi)
= \frac{1}{f_{\mu}(0)} \int_{\mathbb R^n}f_{\mu}(x) dx=\frac{1}{f_{\mu}(0)},
\end{align*}
using \eqref{eq:rho} and integration in spherical coordinates. It is also easily checked,
by direct computation, that for any $\xi\in S^{n-1}$
and any $p>0$ we have \begin{equation}\label{eq:p-moments} \int_{K_{n+p}(\mu)}\langle x,
\xi\rangle_+^p \,dx=\frac{1}{f_{\mu}(0)}\int_{\mathbb R^n}\langle x,
\xi\rangle_+^p\, f_{\mu}(x) \, dx.
\end{equation}
Finally, we need the next inclusion relation between the bodies $K_p(\mu)$.

\begin{lemma}\label{lem:inclusions-between-Kp}
Let $\mu$ be a probability measure on $\mathbb{R}^n$ with bounded density $f_{\mu}$
such that $f_{\mu}(0)>0$. If \, $ 0 < p\ls q$, then
\begin{equation}\label{comparison}K_{p}(\mu)\subseteq
\left(\frac{\|f_{\mu}\|_{\infty}}{f_{\mu}(0)}\right)^{\frac{1}{p}-\frac{1}{q}}K_{q}(\mu).
\end{equation}
\end{lemma}

\begin{proof}The proof is based on the next well-known fact:
If $f:[0,\infty) \to [0,\infty)$ is a bounded integrable function, then
\begin{equation*} F(p):= \left( \frac{p}{\|f\|_\infty} \int_0^\infty x^{p-1}f(x) \, dx\right)^{1/p}
\end{equation*} is an increasing function of $p$ on $(0,\infty)$.
Let us briefly recall the proof of this claim: Without loss of generality we may assume
that $\|f\|_\infty=1$. Then, for any $0<p<q$ and $\gamma>0$, we may write
\begin{align*} \frac{F(q)^q}{q}=\int_0^\infty x^{q-1}f(x)\,
dx&=\int_0^{\gamma}x^{q-1}f(x)\, dx +\int_{\gamma}^\infty x^{q-1}f(x)\, dx\\
\nonumber &\gr  \int_0^{\gamma}x^{q-1}f(x)\, dx+\gamma^{q-p}\int_{\gamma}^\infty
x^{p-1}f(x)\,
dx\\
\nonumber &= \gamma^{q-p}\frac{F(p)^p}{p}-\gamma^q\int_0^1(x^{p-1}-x^{q-1})f(\gamma x)\, dx\\
\nonumber &\gr
\gamma^{q-p}\frac{F(p)^p}{p}-\gamma^q\left(\frac{1}{p}-\frac{1}{q}\right).
\end{align*}The choice $\gamma=F(p)$ minimizes the right hand side and shows that
$F(p)\ls F(q)$.

Using this claim with $f=f_{\mu}$ we see that, for any $q\gr p>0$,
\begin{align*}\rho_{K_q(\mu)}(x) &=\left (\frac{q}{f_{\mu}(0)}
\int_0^{\infty} r^{q-1}f_{\mu}(rx)\, dr \right)^{1/q}
=\left (\frac{\| f_{\mu}\|_{\infty }}{f_{\mu}(0)}\right )^{1/q}\left
(\frac{q}{\| f_{\mu}\|_{\infty }} \int_0^{\infty} r^{q-1}f_{\mu}(rx)\, dr
\right)^{1/q}\\
&=\left (\frac{\| f_{\mu}\|_{\infty }}{f_{\mu}(0)}\right )^{1/q}F(q)
\gr \left (\frac{\| f_{\mu}\|_{\infty }}{f_{\mu}(0)}\right )^{1/q}F(p)
=\left (\frac{\| f_{\mu}\|_{\infty }}{f_{\mu}(0)}\right
)^{1/q-1/p}\left (\frac{\| f_{\mu}\|_{\infty }}{f_{\mu}(0)}\right )^{1/p}F(p)\\
&=\left (\frac{\| f_{\mu}\|_{\infty }}{f_{\mu}(0)}\right
)^{1/q-1/p}\rho_{K_p(\mu)}(x)
\end{align*}
and the lemma follows. \end{proof}

Note that since $\vol_n(K_n(\mu))=1/f_{\mu}(0)>0$, then the above lemma provides that $\vol_n(K_q(\mu))>0$ for every $q \gr n$.

The next lemma establishes a close relation between the family of nonsymmetric $L_p$-centroid
bodies of a probability measure $\mu $ and the family of Ball's sets $K_p(\mu )$. Recall
that for every star-shaped at the origin set $A\subset {\mathbb R}^n$ with $\vol_n(A)>0$ we
denote by $\overline{A}$ the set $\vol_n(A)^{-1/n}A$.

\begin{lemma}\label{lem:passing-to-star-bodies}
Let $\mu$ be a probability measure on $\mathbb{R}^n$ with bounded density $f_{\mu}$
such that $f_{\mu}(0)>0$. For every $p\gr 1$,
\begin{equation*}
Z_p^+(\overline{K_{n+p}}(\mu))\vol_n(K_{n+p}(\mu))^{\frac{1}{p}+\frac{1}{n}}
f_{\mu}(0)^{1/p}=Z_p^+(\mu).\end{equation*}
\end{lemma}

\begin{proof}Let $p\gr 1$. From \eqref{eq:p-moments} we know that
\begin{equation*}
\int_{K_{n+p}(\mu)}\langle x, \xi\rangle_+^p
\,dx=\frac{1}{f_{\mu}(0)}\int_{\mathbb R^n}\langle x, \xi \rangle_+^p\, f_{\mu}(x) \, dx
\end{equation*}for all $\xi\in S^{n-1}$. Since
\begin{equation*}
\int_{K_{n+p}(\mu)}\langle x, \xi\rangle_+^p
\,dx=\vol_n(K_{n+p}(\mu))^{1+\frac{p}{n}}\int_{\overline{K_{n+p}}(\mu)}\langle
x, \xi\rangle_+^p \,dx,\end{equation*}the result follows. \end{proof}

The above discussion reduces the question to obtain a lower bound for the
volume of $Z_p^+(\mu)$ to the corresponding question for the volume of $Z_p^+(K)$
where $K$ is a star-shaped at the origin. When $K$ is a star body, the latter
question has been addressed by Lutwak, Yang and Zhang
in \cite{Lutwak-Yang-Zhang-2000} for the $L_p$-centroid bodies $Z_p(K)$, and later in
the form that we need by Haberl and Schuster in \cite{Haberl-Schuster-2009}. If $K$ is
a star body in ${\mathbb R}^n$
then, for every $1\ls p<\infty $, the body $\mathcal{M}_p^+(K)$ is defined
through its support function
\begin{equation*}h_{\mathcal{M}_p^+(K)}(y)=\left (c_{n,p}(n+p)\int_K\langle x,y\rangle_+^p
dx\right )^{1/p},\end{equation*} where
\begin{equation*}c_{n,p}=\frac{\Gamma\left(\frac{n+p}{2}\right)}{\pi^{\frac{n-1}{2}}\Gamma\left(\frac{p+1}{2}\right)}.\end{equation*}
The normalization of $\mathcal{M}_p^+(K)$ is chosen so that
$\mathcal{M}_p^+(B_2^n)=B_2^n$ for every $p$. Haberl and Schuster \cite[Theorem~6.4]{Haberl-Schuster-2009}
proved that if $K$ is a star body in ${\mathbb R}^n$ then, for every $p\gr 1$,
\begin{equation*}\vol_n(K)^{-\frac{n}{p}-1}\vol_n(\mathcal{M}_p^+(K))\gr \vol_n(B_2^n)^{-\frac{n}{p}}\end{equation*}
with equality if and only if $K$ is a centered ellipsoid in $\mathbb{R}^n$.
Since $\mathcal{M}_p^+(K)=(c_{n,p}(n+p))^{1/p}Z_p^+(K)$, we conclude that if $\vol_n(K)=1$
then
$$\vol_n(Z_p^+(K))^{1/n}=(c_{n,p}(n+p))^{-1/p}\vol_n(\mathcal{M}_p^+(K))^{1/n}\gr \left(\frac{1}{c_{n,p}(n+p)\omega_n}\right)^{1/p}.$$
Taking into account the value of the constant $c_{n,p}$ we can formulate
this result in the language that we use: If $K$ is a star body of volume $1$ in $\mathbb{R}^n$
then
\begin{equation}\label{eq:Zp-star}\vol_n(Z_p^+(K))^{1/n}\gr c\sqrt{p/n}\end{equation}
for every $1\ls p\ls n$, where $c>0$ is an absolute constant.

However, in order to implement the above, we need to ensure that Ball's sets $K_q(\mu)$ are star bodies
for our given measure $\mu$. At this point we shall assume that $\mu$ belongs to the class $\mathcal{P}_n$
of Borel probability measures $\mu$ on $\mathbb{R}^n $ with bounded density $f_\mu$ such that
$K_{f_\mu}=\{f_\mu>0\}$ is convex  with $0$ in its interior and the restriction of $f_\mu$ to $K_{f_\mu}$ is continuous.

\begin{proposition}\label{prop:LYZ}
Let $\mu$ be a probability measure on $\mathbb{R}^n$ that belongs to the class $\mathcal{P}_n$. Then,
\begin{equation*}\vol_n(Z_p^+(\mu))^{1/n}\gr c\|f_{\mu}\|_{\infty}^{-1/n}\sqrt{p/n}\end{equation*}
for every $1\ls p\ls n$, where $c>0$ is an absolute constant.
\end{proposition}

\begin{proof} Let us assume first that $\mu$ is compactly supported. Then, one can easily check that
$K_q(\mu)$ is a star body for every $q>0$. We know that $\vol_n(K_n(\mu))=1/f_{\mu}(0)$. Let $1\ls p\ls n$.
Lemma~\ref{lem:inclusions-between-Kp} shows that
\begin{equation*}K_n(\mu)\subseteq
\left(\frac{\|f_{\mu}\|_{\infty}}{f_{\mu}(0)}\right)^{\frac{p}{n(n+p)}}K_{n+p}(\mu),
\end{equation*}
and hence
$$\vol_n(K_{n+p}(\mu))^{\frac{1}{p}+\frac{1}{n}}\gr \vol_n(K_n(\mu))^{\frac{1}{p}+\frac{1}{n}}\left(\frac{f_{\mu}(0)}{\|f_{\mu}\|_{\infty}}\right)^{\frac{1}{n}}=
f_{\mu}(0)^{-\frac{1}{p}}\|f_{\mu}\|_{\infty}^{-\frac{1}{n}}.$$
Then, Lemma~\ref{lem:passing-to-star-bodies} shows that
\begin{align*}
\vol_n(Z_p^+(\mu))^{1/n} &=\vol_n(Z_p^+(\overline{K_{n+p}}(\mu)))^{1/n}\vol_n(K_{n+p}(\mu))^{\frac{1}{p}+\frac{1}{n}}
f_{\mu}(0)^{1/p}\\
&\gr \|f_{\mu}\|_{\infty}^{-\frac{1}{n}}\vol_n(Z_p^+(\overline{K_{n+p}}(\mu)))^{1/n}.\end{align*}
Using \eqref{eq:Zp-star} we obtain the result.

Now, in the general case of a measure $\mu\in\mathcal{P}_n$, for every $k \in \mathbb{N}$ we define $\nu_k$ to be the
probability measure with density $g_k=\frac{1}{c_k}f_\mu \cdot \mathds{1}_{kB_{2}^n}$, where  $c_k>0$ is a normalization
constant. Note that each $\nu_k$ is compactly supported and belongs to the class $\mathcal{P}_n$.
Notice that $c_k \to 1$, as $k \to \infty$. Then, by the dominated convergence theorem we have that $h_{Z_{p}^+(\nu_k)}(\xi) \to h_{Z_{p}^+(\mu)}(\xi) $, as $k \to \infty$, for every $\xi \in S^{n-1}$. It follows that
$\vol_n(Z_{p}^+(\nu_k)) \to \vol_n(Z_{p}^+(\mu)) $, as $k \to \infty$. Using the lower bound for
$\vol_n(Z_{p}^+(\nu_k))$, we conclude the proof.
\end{proof}

Based on the inclusion of Theorem~\ref{th:deterministic} we can now deduce a lower bound for the volume of the random polytope $K_N$.

\begin{theorem}\label{th:alpha-volume}Let $\beta\in (0,1)$ and $\alpha\gr\frac{1}{2}$. Set
$r(\alpha,\beta):=\frac{2\ln(2e\alpha)}{\beta}$ and $t(\alpha,\beta):=\frac{\beta}{2\ln(2e\alpha)}$.
If $\mu $ is an $\alpha$-regular probability measure on ${\mathbb R}^n$, which belongs to
the class $\mathcal{P}_n$ and has $0$ as a center point, then
for any $(n+1)^{1+r(\alpha,\beta)}\ls N\ls e^n$ we have that
$$\vol_n(K_N)^{1/n}\gr c \sqrt{t(\alpha,\beta)}\|f_{\mu}\|_{\infty}^{-1/n}\frac{\sqrt{\ln(N/n)}}{\sqrt{n}}$$
with probability greater than $1-\exp(-\tfrac{1}{2}N^{1-\beta}n^{\beta})$, where $c>0$ is an absolute constant.
\end{theorem}

\begin{proof}From Theorem~\ref{th:deterministic} we know that if $N\gr c_1(\beta)n$ then
the random polytope $K_N$ satisfies
$$K_N\supseteq T_{\beta\ln\left(\frac{N}{n}\right)}(\mu)$$
with probability greater than $1-\exp(-\tfrac{1}{2}N^{1-\beta}n^{\beta})$.
Moreover, since $\mu$ is $\alpha$-regular and $\ln(1/\varphi_{\mu}(0))\ls\ln (n+1)$,
Proposition~\ref{prop:measure-1} shows that for every $p\gr 2\ln(2e\alpha)\ln (n+1)$ we have that
$$Z_{\frac{p}{2\ln(2e\alpha)}}^+(\mu)\subseteq 2T_p(\mu).$$
Define $r(\alpha,\beta)=\frac{2\ln(2e\alpha)}{\beta}$. If $N\gr (n+1)^{1+r(\alpha,\beta)}$
and $p=\beta\ln\left(\frac{N}{n}\right)$ then
$$\frac{1}{2\ln(2e\alpha)}p\gr\frac{\beta}{2\ln(2e\alpha)}r(\alpha,\beta)\ln (n+1)= \ln(n+1),$$ and hence
$$K_N\supseteq \frac{1}{2} Z_{\frac{\beta}{2\ln(2e\alpha)}\ln\left(\frac{N}{n}\right)}^+(\mu)$$
with probability greater than $1-\exp(-\tfrac{1}{2}N^{1-\beta}n^{\beta})$.
Note that if $(n+1)^{1+r(\alpha,\beta)}\ls N\ls e^n$, then
$$1\ls \ln (n+1)\ls\frac{p}{2\ln(2e\alpha)}=\frac{\beta}{2\ln(2e\alpha)}\ln(N/n)
\ls\ln(N/n)\ls n,$$ and hence we may apply Proposition~\ref{prop:LYZ} to get
$$\vol_n(K_N)^{1/n}\gr c \sqrt{t(\alpha,\beta)}\|f_{\mu}\|_{\infty}^{-1/n}\frac{\sqrt{\ln(N/n)}}{\sqrt{n}}$$
with the same probability, where $t(\alpha,\beta)=\frac{\beta}{2\ln(2e\alpha)}$ and $c>0$ is an absolute constant.
\end{proof}

In the case of log-concave or $s$-concave measures with $s \in \left[-\frac{1}{2n+1},0\right)$, Theorem~\ref{th:alpha-volume}
takes the following form.

\begin{proposition}\label{prop:kconc} Let $\beta \in (0,1)$. If $\mu$ is a centered Borel probability measure on $\mathbb{R}^n$ which is either log-concave or $s$-concave, where $s \in \left[-\frac{1}{2n+1},0\right)$, then for any $c_1(\beta)n \ls N \ls e^n$ we have that
$$\vol_n(K_N)^{1/n}\gr c_2 \sqrt{\beta} \|f_{\mu}\|_{\infty}^{-1/n}\frac{\sqrt{\ln(N/n)}}{\sqrt{n}}$$
with probability greater than $1-\exp(-\tfrac{1}{2}N^{1-\beta}n^{\beta})$, where $c_1(\beta)>0$ is a constant depending only on $\beta$ and $c_2>0$ is an absolute constant.
\end{proposition}

\begin{proof}
It suffices to assume that $\mu$ is $(-\frac{1}{2n+1})$-concave. At first, it is clear that $\mu$ belongs to the class $\mathcal{P}_n$. Then, by \eqref{eq:s-concave} we conlcude that
$$\varphi_\mu(0) \gr \left(1-\frac{1}{2n+1}\right)^{2n+1} \gr \frac{1}{5}.$$
Now, taking into account \eqref{eq:k-Zp-star-2}, an inspection of the proof of Proposition~\ref{prop:measure-1} shows that
$$Z_{p}^+(\mu) \subseteq 2T_{cp}(\mu)$$
where $c>1$ is an absolute constant, for every $1 \ls p \ls n$, since $2\ln(2eC_1p) > \ln(5) \gr \ln(1/\varphi_\mu (0))$.
So, if we set $p=\beta\ln(\frac{N}{n})$, then
$$1 \ls \frac{p}{c} \ls n$$
for any $c_1(\beta)n\ls N \ls e^n$, if $c_1(\beta)$ is chosen large enough. Then, we follow the proof of Theorem~\ref{th:alpha-volume}.
\end{proof}

Assume now that $\mu$ is a centered log-concave probability measure on $\mathbb{R}^n$. It was proved in \cite{DGT1}
that for every $n\ls N\ls e^n$ one has $\vol_n(S_N)^{1/n}\ls c\sqrt{\ln N}/\sqrt{n}$ with probability
greater than $1/N$. We present a variant of the argument which shows a similar upper bound for the
expectation of the volume of $S_N$.

\begin{theorem}\label{th:new-volume}Let $\mu $ be a centered log-concave probability measure
on $\mathbb{R}^n$. Then, for every $n^2\ls N\ls e^n$, one has
\begin{equation}\label{eq:main-new-volume}
c_1\frac{\sqrt{\ln N}}{\sqrt{n}}\ls \mathbb{E}\big(\vol_n(K_N)^{1/n}\big)\ls \mathbb{E}\big(\vol_n(S_N)^{1/n}\big)
\ls c_2\frac{\sqrt{\ln N}}{\sqrt{n}}\end{equation}
where $c_1,c_2>0$ are absolute constants.
\end{theorem}

The starting point is the next general lemma.

\begin{lemma}\label{lem:comp}Let $\mu $ be a probability measure on ${\mathbb R}^n$
such that $Z_p(\mu)$ is a convex body for some $p\gr 1$. Then, for every $\delta>1$ one has
\begin{equation*}{\mathbb E}\,\Big(\sigma (\{\xi : h_{S_N}(\xi )\gr \delta h_{Z_p(\mu )}(\xi
)\})\Big) \ls N\delta^{-p}. \end{equation*}
\end{lemma}

\begin{proof}Let $X$ be a random vector distributed according to $\mu$.
For any $\xi\in S^{n-1}$, Markov's inequality shows that ${\mathbb P}\,(|\langle X,\xi\rangle |\gr\delta
\|\langle \cdot ,\xi\rangle\|_p)\ls \delta^{-p}$. Then,
\begin{align*}{\mathbb P}\,(h_{S_N}(\xi )\gr \delta h_{Z_p(\mu )}(\xi
)) &={\mathbb P}\,(\max_{1\ls j\ls N}|\langle X_j,\xi \rangle
|\gr\delta\|\langle \cdot ,\xi\rangle\|_p)\\
\nonumber &\ls N\,{\mathbb P}\,(|\langle X,\xi\rangle |\gr\delta
\|\langle \cdot ,\xi\rangle\|_p)\ls N\delta^{-p}.\end{align*}
Then,
\begin{equation*}{\mathbb E}\,\Big(\sigma (\{\xi : h_{K_N}(\xi )\gr \delta h_{Z_p(\mu )}(\xi
)\})\Big)=\int_{S^{n-1}}{\mathbb P}\,(h_{K_N}(\xi )\gr \delta
h_{Z_p(\mu )}(\xi ))\,d\sigma (\xi )\ls N\delta^{-p}\end{equation*} by
Fubini's theorem. \end{proof}

Now, we make the additional assumption that $\mu$ is centered and log-concave. In what follows,
for every symmetric convex body $K$ in $\mathbb{R}^n$ and for any $q\neq 0$ we define
\begin{equation*}w_q(K)=\left (\int_{S^{n-1}}h_K(\xi)^qd\sigma (\xi)\right )^{1/q}.
\end{equation*} Note that $w_1(K)=w(K)$ is the mean width of $K$. The parameters $w_q(K)$, $q\gr 1$
were introduced and studied by Litvak, Milman and Schechtman in \cite{Litvak-VMilman-Schechtman-1998}.

\begin{proof}[Proof of Theorem~$\ref{th:new-volume}$]We may assume that $\mu$ is
isotropic. Set $p=\ln N \ls n$. We start with the observation that
\begin{equation}\label{eq:new-volume-1}\vol_n(K_N)^{1/n}\ls \vol_n(S_N)^{1/n}\ls \frac{c_1}{\sqrt{n}}
w_{-p}(S_N)\end{equation} for some absolute constant $c_1>0$. Indeed, using H\"{o}lder's inequality we write
$$\vrad(S_N^{\circ }) =\left (\int_{S^{n-1}}\frac{1}{h_{S_N}^n(\xi )}\,d\sigma
(\xi )\right )^{1/n}\gr \left
(\int_{S^{n-1}}\frac{1}{h_{S_N}^p(\xi )}\,d\sigma (\xi )\right
)^{1/p}=\frac{1}{w_{-p}(S_N)}.$$ Then, the
Blaschke-Santal\'{o} inequality (see~\cite[Theorem~1.5.10]{AGA-book}) implies that
\begin{equation*}\vol_n(S_N)^{1/n}\approx \frac{1}{\sqrt{n}}\,\vrad(S_N)\ls \frac{1}{\sqrt{n}}\,\vrad(S_N^{\circ })^{-1}\ls
\frac{c_1w_{-p}(S_N)}{\sqrt{n}}.\end{equation*}
Next, we write
\begin{align*}
w_{-p/4}(Z_p(\mu ))^{-p/2} &= \left
(\int_{S^{n-1}}\frac{1}{h^{p/4}_{Z_p(\mu )}(\xi )}\,d\sigma
(\xi )\right )^2\\
&\ls  \left (\int_{S^{n-1}}\frac{1}{h^{p/2}_{S_N}(\xi )}\,d\sigma
(\xi )\right )\left (\int_{S^{n-1}}\frac{h^{p/2}_{S_N}(\xi
)}{h^{p/2}_{Z_p(\mu )}(\xi )}\,d\sigma (\xi )\right ),
\end{align*}
which can be rewritten as
\begin{equation}\label{eq:new-volume-2}w_{-p/2}(S_N)=\left (\int_{S^{n-1}}\frac{1}{h^{p/2}_{S_N}(\xi )}\,d\sigma
(\xi )\right )^{-2/p}\ls w_{-p/4}(Z_p(\mu )) \left (\int_{S^{n-1}}\frac{h^{p/2}_{S_N}(\xi
)}{h^{p/2}_{Z_p(\mu )}(\xi )}\,d\sigma (\xi )\right )^{2/p}.\end{equation}
Now, we estimate the integral
\begin{equation}\label{eq:new-volume-3}\int_{S^{n-1}}\frac{h^{p/2}_{S_N}(\xi
)}{h^{p/2}_{Z_p(\mu )}(\xi )}\,d\sigma (\xi
)=\int_0^{\infty }\frac{p}{2}t^{\frac{p}{2}-1}\big(\sigma \big (\xi :\,h_{S_N}(\xi
)\gr th_{Z_p(\mu )}(\xi )\big )\big)\,dt. \end{equation}
Taking expectations in \ref{eq:new-volume-3} and using Lemma~\ref{lem:comp}, we see that
$${\mathbb E}\,\left(\int_{S^{n-1}}\frac{h^{p/2}_{S_N}(\xi)}{h^{p/2}_{Z_p(K)}(\xi )}\,d\sigma (\xi )\right)
\ls e^{p/2}+\int_e^{\infty }\frac{p}{2}t^{\frac{p}{2}-1}Nt^{-p}\,dt= e^{p/2}+ Ne^{-p/2}=2e^{p/2}.$$
Going back to \eqref{eq:new-volume-2} we get
\begin{align}\label{eq:new-volume-4}\mathbb{E}(w_{-p/2}(S_N)) &\ls
w_{-p/4}(Z_p(\mu ))\,\mathbb{E}\Bigg(\Bigg (\int_{S^{n-1}}\frac{h^{p/2}_{S_N}(\xi
)}{h^{p/2}_{Z_p(\mu )}(\xi )}\,d\sigma (\xi )\Bigg )^{2/p}\Bigg)\\
\nonumber &\ls w_{-p/4}(Z_p(\mu ))\Bigg(\mathbb{E}\Bigg (\int_{S^{n-1}}\frac{h^{p/2}_{S_N}(\xi
)}{h^{p/2}_{Z_p(\mu )}(\xi )}\,d\sigma (\xi )\Bigg)\Bigg)^{2/p} \ls c_2w_{-p/4}(Z_p(\mu ))
\end{align}
where $c_2>0$ is an absolute constant. Next, recall that if $\mu $ is a log-concave probability measure on
${\mathbb R}^n$ then, for any $1\ls q\ls n-1$,
\begin{equation*}w_{-q}(Z_q(\mu ))\approx\frac{\sqrt{q}}{\sqrt{n}}\,I_{-q}(\mu )
\end{equation*} where $I_q(\mu )=\left (\int_{{\mathbb R}^n}|x|^q\,d\mu (x)\right )^{1/q}$ for $0\neq q>-n$.
This is a result of Paouris from~\cite{Paouris-2012}; see also \cite[Theorem~5.3.16]{BGVV-book}. Note that $1\ls p/4\ls n-1$ because we have assumed that $n^2\ls N\ls e^{n}$. Since $Z_p(\mu )\subseteq
c_3Z_{p/4}(\mu )$ for an absolute constant $c_3>0$, we can write
\begin{equation*}w_{-p/4}(Z_p(\mu ))\ls c_4w_{-p/4}(Z_{p/4}(\mu ))\ls
\frac{c_5\sqrt{p}}{\sqrt{n}}\,I_{-p/4}(\mu ).\end{equation*}
Since we have assumed that
$\mu $ is isotropic, we have $I_{-p/4}(\mu )\ls I_2(\mu )=\sqrt{n}$, and it follows that
\begin{equation*}w_{-p/4}(Z_p(\mu ))\ls c_6\sqrt{p}.\end{equation*}
Combining the last inequality with \eqref{eq:new-volume-1} and \eqref{eq:new-volume-4}  we have
\begin{equation*}\mathbb{E}\big(\vol_n(K_N)^{1/n}\big)\ls \mathbb{E}\big(\vol_n(S_N)^{1/n}\big)\ls \frac{c_7\sqrt{p}}{\sqrt{n}}=
c_7\frac{\sqrt{\ln N}}{\sqrt{n}}.\end{equation*}
The lower bound in \eqref{eq:main-new-volume} is an immediate consequence of Proposition~\ref{prop:kconc} for $\beta=1/2$ and Markov's inequality, taking also into account that $\|f_\mu\|^{1/n} \approx 1$, as $\mu$ is isotropic, and $\ln(\frac{N}{n}) \gr \frac{1}{2}\ln N$, since $n^2 \ls N \ls e^n$.
\end{proof}

%%%%%%%%%%% End of paper body %%%%%%%%%%%%%%%%%%%%%%%%%%%%%%%
\bigskip

\noindent {\bf Acknowledgement.} We would like to thank S.~Brazitikos and A.~Giannopoulos
for helpful discussions. The second named author acknowledges support by a PhD scholarship
from the National Technical University of Athens.

\bigskip

%%%%%%%%%%%%%%%%%%%%%%%%%%%%%%%%%%%%%%%%%%%%%%%%%%%%%%%%%%%%%%%%%%%%%%%%%
%%%%%%%%%%%%%%%%%%%%%%%%%%%%%%%%%%%%%%%%%%%%%%%%%%%%%%%%%%%%%%%%%%%%%%%%%

\footnotesize
\bibliographystyle{amsplain}

\small

\bigskip

\medskip

\thanks{\noindent {\bf Keywords:} Random polytopes, Tukey half-space depth, floating bodies, heavy tails, log-concave measures.

\smallskip

\thanks{\noindent {\bf 2020 MSC:} Primary 60D05; Secondary 52A22, 52A23, 60E15, 62H05.}

\bigskip

\bigskip

\noindent \textsc{Minas \ Pafis}: Department of Mathematics, National and Kapodistrian University of Athens, Panepistimioupolis 157-84,
Athens, Greece.

\smallskip

\noindent \textit{E-mail:} \texttt{mipafis@math.uoa.gr}

\bigskip

\noindent \textsc{Natalia \ Tziotziou}: School of Applied Mathematical and Physical Sciences, National Technical University of Athens, Department of Mathematics, Zografou Campus, GR-157 80, Athens, Greece.

\smallskip

\noindent \textit{E-mail:} \texttt{nataliatz99@gmail.com}

\end{document}